\documentclass[10pt]{article}

\usepackage[a4paper,margin=.879in,footskip=0.25in]{geometry}

\usepackage{amsmath,amsthm} 
\usepackage{amssymb}
\usepackage[all]{xypic} 
\usepackage{amsbsy}
\usepackage{graphicx}  
\usepackage{subfigure} 
\usepackage[all,cmtip]{xy} 
\usepackage{multirow}  
\usepackage{multicol}
\usepackage{rotating} 
\usepackage{comment} 
\usepackage{tikz} 
\usetikzlibrary{arrows, arrows.meta, decorations.markings, decorations.pathmorphing, shapes} 
\usepackage{relsize} 
\usepackage{hyperref} 
\usepackage{xcolor} 
\usepackage{lscape} 
\usepackage{manfnt} 
\usepackage{fancyvrb}  
 
\numberwithin{equation}{section}

\newtheorem{thm}{Theorem}[section]
\newtheorem{cor}[thm]{Corollary}

\newtheorem{lem}[thm]{Lemma}
\newtheorem{prop}[thm]{Proposition}
\theoremstyle{definition} 
\newtheorem{defn}[thm]{Definition}
\newtheorem{rem}[thm]{Remark}
\newtheorem{exam}[thm]{Example}

\newcommand{\even}{\text{even}}

\begin{document}

\title{Multiplicative preprojective algebras of Dynkin quivers}
\author{Daniel Kaplan}
\date{}

\maketitle

\begin{abstract}
For a commutative ring $R$ and an ADE Dynkin quiver $Q$, we
prove that the multiplicative preprojective algebra of Crawley-Boevey and
Shaw, with parameter $q=1$, is isomorphic to the (additive) preprojective
algebra as $R$-algebras if and only if the bad primes for $Q$ -- 2 in type
D, 2 and 3 for $Q = E_{6}$, $E_{7}$ and 2, 3 and 5 for $Q= E_{8}$ -- are
invertible in $R$. We construct an explicit isomorphism over
$\mathbb{Z}[1/2]$ in type D, over $\mathbb{Z}[1/2, 1/3]$ for
$Q = E_{6}$, $E_{7}$ and over $\mathbb{Z}[1/2, 1/3, 1/5]$ for
$Q=E_{8}$. Conversely, if some bad prime is not invertible in $R$, we show
that the additive and multiplicative preprojective algebras differ in zeroth
Hochschild homology, and hence are not isomorphic. In fact, one only needs
the vanishing of certain classes in zeroth Hochschild homology of the multiplicative
preprojective algebra, utilizing a rigidification argument for isomorphisms
that may be of independent interest.

In the setting of Ginzburg dg-algebras, our obstructions are new in type
E and give a more elementary proof of the negative result of Etg\"{u}--Lekili
\cite[Theorem 13]{EL17} in type D. Moreover, the zeroth Hochschild homology
of the multiplicative preprojective algebra, computed in Section~\ref{sec:_HH_0}, can be interpreted as the space of unobstructed deformations
of the multiplicative Ginzburg dg-algebra by Van den Bergh duality. Finally,
we observe that the multiplicative preprojective algebra
is not symmetric Frobenius if $Q \neq A_1$, a departure from the additive preprojective algebra in characteristic 2 for
$Q = D_{2n}$, $n \geq 2$ and $Q =E_{7}$, $E_{8}$.
\end{abstract}

\section{Introduction}
\label{sec1}

Fix a quiver $Q$, a field $k$, and write $\text{char}(k)$ for the characteristic
of $k$. Crawley-Boevey and Shaw defined the multiplicative preprojective
algebra of $Q$ over $k$, denoted $\Lambda _{k}(Q)$, to study Kac's middle
convolution operation and the Deligne--Simpson problem in
\cite{CBS06}. These algebras are a multiplicative version of the preprojective
algebra, denoted $\Pi _{k}(Q)$, as can be made precise using a moment-map
interpretation of the relations \cite[Theorem 6.7.1]{VdB_Double}. Yamakawa
developed the theory of multiplicative quiver varieties by looking at spaces
of semistable representations of the multiplicative preprojective algebra,
in analogy to Nakajima's quiver varieties for preprojective algebras
\cite{Yamakawa08}. Furthermore, the differential-graded multiplicative
preprojective algebra has recently arisen in the study of certain wrapped
Fukaya categories \cite{EL17,EL19} and microlocal sheaves
\cite{BK16}. Despite this attention in representation theory and geometry,
the multiplicative preprojective algebra is less understood algebraically
than the preprojective algebra. Hence the aim of this paper is to establish
elementary properties of multiplicative preprojective algebras, as part
of a larger program: \cite{CBS06}, \cite{Shaw05}, \cite{KS20}.

If $Q'$ is a quiver obtained from the quiver $Q$ by changing the orientation
of a single arrow, then $\Lambda _{k}(Q) \cong \Lambda _{k}(Q')$. Hence
the behavior the multiplicative preprojective algebra only depends on the
graph underlying $Q$. In fact, the behavior of these algebras depend heavily
on the trichotomy of the underlying graph into (1) simply-laced Dynkin
diagrams, (2) simply-laced extended Dynkin diagrams, and (3) others. In
this article, we restrict to the case of simply-laced Dynkin diagrams:
two infinite families $\{ A_{n} \}_{n \in \mathbb{N}}$,
$\{ D_{n} \}_{n \geq 4}$ and three exceptions $E_{6}, E_{7}, E_{8}$, see
Fig.~\ref{fig:Dynkin_diagrams}.
%
\begin{figure}
\begin{equation*}
\fontsize{10.99}{10}\selectfont
\xymatrix@R=1em{ & & & & 1 \ar[rd]^a & & & & \\
A_{n} = 1 \ar[r]^-{a_{1}} & 2 \ar[r]^-{a_{2}} & \cdots \ar[r]^{a_{n-1}} & n & D_{n} = & 3 & 4 \ar[l]_{c_{1}} & \ar [l]_{c_{2}} \cdots & \ar [l]_{c_{n-3}} n \\
& & & & 2 \ar[ru]_b & & & &}
\end{equation*}
\begin{equation*}\fontsize{10.99}{10}\selectfont
\begin{array}{l}E_{6, 7, 8} =
\xymatrix{ 1 \ar[r]^d & 2 \ar[r]^b & 3 & 5 \ar[l]_c & 6 \ar[l]_e & 7 \ar[l]_f & 8 \ar[l]_g \\ & & 4 \ar[u]^a & & & & }
\end{array}
\end{equation*}\vspace{-5pt}

\caption{Quivers corresponding to simply-laced Dynkin diagrams with a labelling
of the vertices and arrows and the arrows oriented towards the central vertex.}
\label{fig:Dynkin_diagrams}
\end{figure}
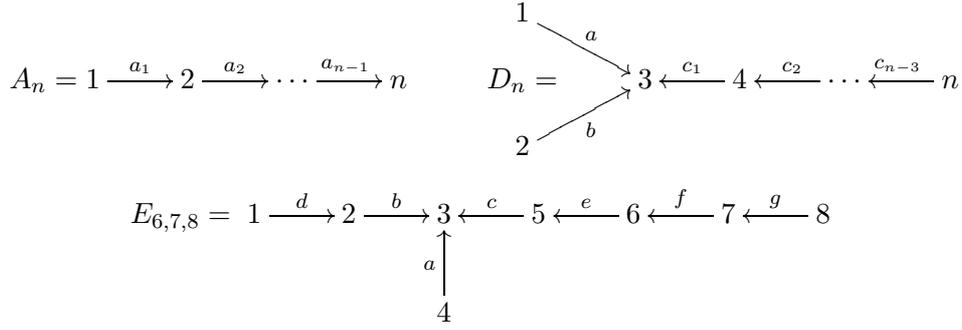

Crawley-Boevey non-constructively proved that
$\Lambda _{k}(Q) \cong \Pi _{k}(Q)$ for $Q$ ADE, over
$k= \mathbb{C}$ in \cite[Corollary 1]{CB13}. However, Shaw proved that
$\Lambda _{k}(D_{4}) \ncong \Pi _{k}(D_{4})$ if
$\text{char}(k)=2$, \cite[Lemma 5.2.1]{Shaw05}. This leaves the natural
question:
\begin{itemize}
\item[]
\textbf{For which $(Q, k)$, with $Q$ an ADE quiver and $k$ a field, is
$\Lambda _{k}(Q) \cong \Pi _{k}(Q)$?}
\end{itemize}

Malkin, Ostrik, and Vybornov in \cite{MOV06} give the following notion:
%
\begin{defn}
\label{def:_bad_prime}
Define $p$ to be a \emph{bad} prime for the quiver $Q$ if:
\begin{equation*}
\begin{cases}
p = 2 & \text{for } Q = D_{n}, \ n\geq 4
\\
p = 2, 3 & \text{for } Q= E_{6}, \hspace{.05cm} E_{7}
\\
p = 2, 3, 5 & \text{for } Q=E_{8}.
\end{cases}
\end{equation*}
Define $p$ to be \emph{good} for the quiver $Q$ otherwise.
\end{defn}

\begin{thm}
\label{thm:_main}
Let $Q$ be an ADE Dynkin. $\Lambda _{k}(Q) \cong \Pi _{k}(Q)$ if and only
if $\text{char}(k)$ is good for $Q$.
\end{thm}

Note that the definitions of the multiplicative and additive preprojective
algebra make sense over a general commutative ring, $R$. We obtain this
result over $R$, where the condition that the $\text{char}(k)$ is good for
$Q$ is modified to the condition that all bad primes for $Q$ are invertible
in $R$.

To prove this theorem, we construct explicit isomorphisms over
$\mathbb{Z}[1/2]$ in type D, over $\mathbb{Z}[1/2, 1/3]$ for
$E_{6}$, $E_{7}$ and over $\mathbb{Z}[1/2, 1/3, 1/5]$ for $E_{8}$. For type
D, using the orientation and labels in Fig.~\ref{fig:Dynkin_diagrams} and writing
$\gamma _{i} := c_{i}^{*} c_{i}$ we have the isomorphism
$\phi : \Lambda _{\mathbb{Z}[1/2]}(D_{n}) \rightarrow \Pi _{
\mathbb{Z}[1/2]}(D_{n})$ defined to be the identity on all vertices and
arrows except:
\begin{equation*}
\phi (a) = a p(\gamma _{1}) \quad \quad \phi (a^{*}) = q(\gamma _{1}) a^{*}
\quad \quad \phi (c_{i}) = c_{i} p(\gamma _{i})
\end{equation*}
where $q(x) := 1 +x/2$ and $p(x) := q^{-1}(x)$. For type E, the isomorphisms
are explicit but less presentable.

While the construction of these isomorphisms is new, their existence is
known. Etg\"{u}--Lekili, working at the level of dg-algebras
$\mathcal{G}^{\text{mult}}_{Q}$ and $\mathcal{G}^{\text{add}}_{Q}$, whose
zeroth homology is respectively $\Lambda _{k}(Q)$ and $\Pi _{k}(Q)$, prove
the existence of a quasi-isomorphism in type D if and only if
$\text{char}(k) \neq 2$, \cite[Theorem 13]{EL17}. The technique involves
realizing $\mathcal{G}^{\text{mult}}_{Q}$ as a flat, filtered deformation
of $\mathcal{G}^{\text{add}}_{Q}$ and then computing that the relevant deformation
class in $\text{HH}^{2}(\mathcal{G}^{\text{add}}_{Q})$ is zero if and only
if $\text{char}(k) \neq 2$. More recently, Lekili--Ueda computed
$\text{HH}^{\bullet}(\mathcal{G}^{\text{add}}_{Q})$ for $Q$ ADE Dynkin in
\cite[Section 5]{LU21}. In particular, they computed that the positive
piece of $\text{HH}^{2}(\mathcal{G}^{\text{add}}_{Q})$ is zero if
$\text{char}(k)$ is good for $Q$. Hence at the level of algebras they prove
the existence of an isomorphism $\Lambda _{k}(Q) \cong \Pi _{k}(Q)$ in
good characteristic. In this setting, our main result says that the relevant
deformation class in $\text{HH}^{2}(\mathcal{G}^{\text{add}}_{Q})$ is non-trivial
both for $Q$ type E if $\text{char}(k) = 3$ and for $Q =E_{8}$ if
$\text{char}(k) = 5$. Hence $\mathcal{G}^{\text{mult}}_{Q}$ and
$\mathcal{G}^{\text{add}}_{Q}$ are \emph{not} quasi-isomorphic over bad characteristic,
completing the proof of the following:
%
\begin{thm} [Theorem~\ref{thm:_dg_main}]
\label{thm1.3}
Let $Q$ be ADE Dynkin.
$\mathcal{G}^{\text{mult}}_{Q} \cong \mathcal{G}^{\text{add}}_{Q}$ are quasi-isomorphic
if and only if $\text{char}(k)$ is good for $Q$.
\end{thm}

For any $k$-algebra, $A$, define the zeroth Hochschild homology of $A$ to be
\begin{equation*}
\text{HH}_{0}(A) := A/[ A, A] =: A_{\text{cyc}}
\end{equation*}
where $[A, A]$ denotes the $k$-linear span
of the set of commutators in $A$. We show for $Q$ ADE Dynkin
that $\Lambda _{k}(Q)_{\text{cyc}} = kQ_{0}$, whereas $\Pi_k(Q)_{\text{cyc}} \neq k Q_0$ in bad characteristic following Schedler \cite[Theorem 13.1.1]{Schedler16}. Hence
$\Lambda _{k}(Q)_{\text{cyc}} \cong \Pi _{k}(Q)_{\text{cyc}}$, if and only
if $\text{char}(k)$ is good for $Q$. Note if $\text{char}(k)$ is bad for
$Q$ then the discrepancy in zeroth Hochschild homology gives an obstruction
to both an isomorphism $\Lambda _{k}(Q) \cong \Pi _{k}(Q)$ and a quasi-isomorphism
$\mathcal{G}^{\text{mult}}_{Q} \cong \mathcal{G}^{\text{add}}_{Q}$.

Alternatively, we give a procedure to utilize automorphisms of
$\Pi _{k}(Q)$ to \emph{correct} any $k$-algebra isomorphism
$\Lambda _{k}(Q) \rightarrow \Pi _{k}(Q)$ to an isomorphism that is the
identity on vertices and takes certain paths to themselves plus higher
order terms. Hence one need not compute the entire zeroth Hochschild homology.
Rather, it suffices to find a path $p$ that is zero when viewed as an element
of $\Lambda _{k}(Q)_{\text{cyc}}$ but non-zero in
$\Pi _{k}(Q)_{\text{cyc}}$. And indeed -- in the case $Q = D_{4}$ with arrows
$a, b, c$ pointing towards the central vertex -- $c^{*} c b^{*}b$ is such
a path. Notice the relation in $\Lambda_k(D_{4})$ can be written:
\begin{equation*}
[a^{*}, a] + [b^{*}, b] + [c^{*}, c] + c^{*} c b^{*}b = 0
\end{equation*}
realizing
$c^{*} c b^{*}b \in [\Lambda _{k}(D_{4}), \Lambda _{k}(D_{4})]$. Whereas,
$c^{*} c b^{*}b \notin [\Pi _{k}(D_{4}), \Pi _{k}(D_{4})]$ if
$\text{char}(k) =2$, see Lemma~\ref{lem:_length_4}. In a similar vein, in Lemma~\ref{lem:_length_6} we show the length 6 path
$b^{*} b a^{*} a b^{*} b$ is zero in
$\Lambda _{k}(E_{6})_{\text{cyc}}$ and non-zero in
$\Pi _{k}(E_{6})_{\text{cyc}}$ if $\text{char}(k) = 3$. Moreover, in Lemma~\ref{lem:_length_10} we show the length 10 path
$b^{*}b a^{*} a b^{*} b a^{*} a b^{*} b$ is zero in
$\Lambda _{k}(E_{8})_{\text{cyc}}$ and non-zero in
$\Pi _{k}(E_{8})_{\text{cyc}}$ if $\text{char}(k) = 5$.

As a final application, we demonstrate that $\Lambda _{k}(Q)$ and
$\Pi _{k}(Q)$, when non-isomorphic, can have different structure. Namely,
if $\text{char}(k) = 2$ and $Q = D_{2n}, E_{7}, E_{8}$ then
$\Pi _{k}(Q)$ is symmetric Frobenius with form given by summing the coefficients
of the top degree cycles. Whereas, $\Lambda _{k}(Q)$ is not symmetric Frobenius
as any symmetric form on $\Lambda _{k}(Q)$ factors though
$\Lambda _{k}(Q)_{\text{cyc}} = k Q_{0}$, and hence is degenerate. In good
characteristic, our isomorphisms restrict to the identity in top degree,
establishing that $\Lambda _{k}(Q)$ is Frobenius with form
$\Lambda _{k}(Q) \rightarrow \Lambda ^{\text{top}}_{k}(Q) \rightarrow k$
projecting to top degree, as explained in Section~\ref{ss:_Frobenius}.

\bigskip
\noindent
\textbf{Open questions}

Here are several natural questions not addressed in this work:
\begin{itemize}
\item[(1)] Do the isomorphisms $\Lambda _{k}(Q) \cong \Pi _{k}(Q)$ lift
to quasi-isomorphisms
$\mathcal{G}_{Q}^{\text{mult}} \cong \mathcal{G}_{Q}^{\text{add}}$?
\item[(2)] Is the multiplicative preprojective algebra with parameter
$q$, $\Lambda ^{q}_{k}(Q)$, isomorphic to $\Pi _{k}^{\lambda}(Q)$ for some
deformation parameter $\lambda $ depending on $q$, if and only if
$\text{char}(k)$ is good for $Q$?
\item[(3)] Can one build a uniform isomorphism
$\Lambda _{k}(Q) \cong \Pi _{k}(Q)$ for each ADE Dynkin quiver?
\end{itemize}
Note for (2) that one should impose the condition that
$\Pi ^{\lambda}_{k}(Q)$ is a flat deformation of $\Pi _{k}(Q)$, in order
to avoid instances where both algebras are zero or concentrated in degree
zero. In this case, we checked that the relevant classes in zeroth Hochschild
homology differ if $\text{char}(k) = 2$ or $3$.

\bigskip
\noindent
\textbf{Non-Dynkin case}

We suspect for $Q$ connected, not ADE Dynkin that
$\Pi _{R}(Q) \ncong \Lambda _{R}(Q)$ as $R$-algebras unless $R$ is
the trivial ring. Etg\"{u}--Lekili claim that their representation varieties
are not isomorphic (see the paragraph in the introduction between Remark
4 and Example 5 in \cite{EL19}). We outline an argument for $Q_{E}$ extended
Dynkin and for $Q'$ containing a cycle.

First note that it suffices to prove the result for $R = k$ a field since
any isomorphism $\Pi _{R}(Q) \cong \Lambda _{R}(Q)$ will descend to one
$\Pi _{R/ \mathfrak{m}}(Q) \cong \Lambda _{R/ \mathfrak{m}}(Q)$ where
$\mathfrak{m}$ is a maximal ideal in $R$.

If $Q'$ contains an unoriented cycle, including the extended Dynkin case
$\tilde{A}_{n}$, then $\Pi _{k}(Q') \ncong \Lambda _{k}(Q')$ should
follow from an explicit $k$-basis for $\Lambda _{k}(Q)$ given in
\cite[Proposition 7.11]{KS20}.

For each extended Dynkin quiver $Q_{E}$ with extended vertex 1, Shaw in
\cite[Theorem 4.1.1]{Shaw05} found a polynomial
$g(Q_{E}) \in k[X, Y, Z]$ such that:
\begin{equation*}
e_{1} \Lambda _{k}(Q_{E}) e_{1} \cong k[X, Y, Z]/(g(Q_{E})).
\end{equation*}
Earlier work of Crawley-Boevey and Holland \cite{CBH98} establishes that
the du Val polynomial $f(Q_{E}) \in k[X, Y, Z]$ satisfies
\begin{equation*}
e_{1} \Pi _{k}( Q_{E}) e_{1} \cong k[X, Y, Z]/(f(Q_{E})).
\end{equation*}
Now for $Q_{E} \neq \tilde{A}_{n}$, $Q_{E}$ is acyclic so Corollary~\ref{cor:_correct_iso} says any isomorphism
$\Lambda _{k}(Q_{E}) \rightarrow \Pi _{k}(Q_{E})$ can be corrected to one
sending
$e_{1} \Lambda _{k}(Q_{E}) e_{1} \rightarrow e_{1} \Pi _{R}( Q_{E}) e_{1}$
and sending each arrow to itself plus higher order terms. Hence it induces
an isomorphism
$k[X, Y, Z]/(g(Q_{E})) \rightarrow k[X, Y, Z]/(f(Q_{E}))$ taking each variable
to itself plus higher order terms. One can check that no such isomorphism
exists, although note that (at least in type E) an isomorphism exists after
completing both algebras at the origin.

\bigskip
\noindent
\textbf{Literature independence}

To highlight the elementary nature of these arguments, notice that one
can arrive at the main result, Theorem~\ref{thm:_main}, using exclusively
internal results and Schedler's result
\cite[Theorem 13.1.1]{Schedler16}. In fact, we need only that a single
torsion class in each of $\Pi _{\mathbb{Z}}(D_{4})_{\text{cyc}}$,
$\Pi _{\mathbb{Z}}(E_{6})_{\text{cyc}}$, and
$\Pi _{\mathbb{Z}}(E_{8})_{\text{cyc}}$ is non-zero. This is itself independent
of the literature, utilizing only the Diamond Lemma for modules over a
commutative ring in \cite[Appendix A]{Schedler16}.

\bigskip
\noindent
\textbf{Computer dependence}

Four results rely on computer assistance: Lemma~\ref{lem:_length_10} and
Propositions~\ref{prop:_HH_0_Em}, \ref{prop:_iso_E7}, and \ref{prop:_iso_E8}. Without a computer, one can still prove all results
in Sections~\ref{sec:_background} and \ref{sec:_correct_iso}, find obstructions
in characteristic 2 and 3, and construct isomorphisms in type D and for
$E_{6}$.

\bigskip
\noindent
\textbf{Organization of the paper}

In Section~\ref{sec:_background} we fix notation, define the relevant objects
of study, and prove some elementary results. In Section~\ref{sec:_correct_iso}, we explain how any $k$-algebra isomorphism can
be corrected by utilizing certain automorphisms of the target, to one preserving
the vertices. In Section~\ref{sec:_HH_0}, we demonstrate the non-existence
of an isomorphism in bad characteristic by computing the zeroth Hochschild
homology of the multiplicative preprojective algebra of Dynkin quivers.
In Section~\ref{sec:_construct_isos}, we construct an isomorphism in good
characteristic, completing the proof of Theorem~\ref{thm:_main}. In Section~\ref{ss:_dg_results} we prove Theorem~\ref{thm:_dg_main} and interpret
the Hochschild homology computation in terms of Hochschild cohomology using
Van den Bergh duality. Section~\ref{ss:_Frobenius} proves the non-existence
of a symmetric Frobenius structure on the multiplicative preprojective algebra if $Q \neq A_1$. We conclude with an appendix containing Magma code to
verify our constructed maps descend to isomorphisms.

Fig.~\ref{fig:_flowchart} is a flowchart of implications. In particular
it shows that Section~\ref{sec:_correct_iso} is extraneous, but using it
simplifies the work needed in Section~\ref{sec:_HH_0} to arrive at the
main result. And note that Section~\ref{sec:_correct_iso} explains how
the isomorphisms are constructed in Section~\ref{sec:_construct_isos}.

\fboxsep=1.8pt
\begin{figure}
\hspace*{5pt}
\xymatrix@R=.00em@C=.65em{
& \fbox{\txt{\underline{Section 3} \\Correct Isomorphisms}} \ar@{=>}[r] & \fbox{\txt{\underline{Section 4} \textcolor{red}{(only lemmas)} \\ Zeroth Hochschild Homology}} \ar@{=>}[r] \ar@{=>}!<7ex,-1ex>;[rdd]!<0ex,0ex> &
\fbox{\txt{\underline{Section 6.1} \\ Theorem~\ref{thm:_dg_main} }}\\ \\
\fbox{\txt{\underline{Section 2} \\Background}} \ar@{=>}[ruu] \ar@{=>}[r] \ar@{=>}[rdd]!<-9ex,0ex>
& \fbox{\txt{\underline{Section 7} \\Appendix}} \ar@{=>}[r] & \fbox{\txt{\underline{Section 5} \\Construct Isomorphisms}} \ar@{=>}!<7ex,-1ex>;[rdd]!<0ex,0ex> \ar@{=>}[r]&
\fbox{\txt{\underline{Main result} \\ Theorem~\ref{thm:_main} }} \\ \\
& \fbox{\txt{\underline{Section 4} \\ Zeroth Hochschild Homology}} \ar@{=>}[rr] \ar@{=>}[rd]& &
\fbox{\txt{\underline{Main result} \\ Theorem~\ref{thm:_main} }} \\
& & \fbox{\txt{\underline{Section 6.2} \\ Frobenius Structures}}
}\vspace{10pt}
\caption{A flowchart of implications, where $\fbox{x} \implies \fbox{y}$ means
$x$ is a prerequisite to $y$.}
\label{fig:_flowchart}
\end{figure}
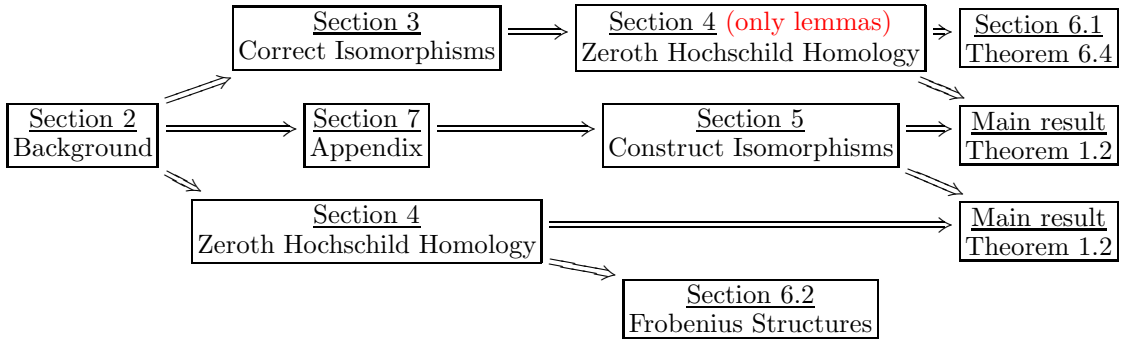

\subsection*{Acknowledgements}
This work was funded by the Engineering and Physical Sciences Research Council (EPSRC) grant number EP/S03062X/1. The Geometry and Mathematical Physics
and Algebra groups at the University of Birmingham provided an engaging
research environment. The work benefited from conversations and emails
with Travis Schedler, who caught a mistake in an earlier draft of this
paper. Thank you Yanki Lekili for explaining your result with Etg\"{u}
in a cab to the airport in Luminy and for making me aware of your work
with Ueda. The exposition evolved following discussions with Tyler Kelly,
Jack Smith, and Michael Wong. David Craven was extremely helpful with my
questions regarding Magma implementation and Jonathan Kaplan provided general
coding suggestions. Finally, thank you to the anonymous
referee for carefully reading an earlier version of this paper and catching
several errors.

\section{Background}
\label{sec:_background}

Let $Q$ be a quiver (i.e., a directed graph) with vertex set $Q_{0}$, arrow
set $Q_{1}$, and source and target maps
$s, t:Q_{1} \rightarrow Q_{0}$ respectively. An \emph{automorphism} of
$Q$ is a pair of bijections
$(\varphi _{0}: Q_{0} \rightarrow Q_{0}, \varphi _{1}: Q_{1}
\rightarrow Q_{1})$ satisfying
$\varphi _{0} \circ s = s \circ \varphi _{1}$ and
$\varphi _{0} \circ t = t \circ \varphi _{1}$. Let $Q^{\text{op}}$ denote
the opposite quiver (i.e., $Q^{\text{op}}$ has the same underlying graph
as $Q$ but each arrow is pointing in the opposite direction). For an arrow
$a \in Q_{1}$, we write $a^{*}$ for the reverse arrow in
$Q_{1}^{\text{op}}$. Let $\overline{Q}$ denote the doubled quiver formed
from $Q$ by adding the arrow $a^{*}$ for all $a \in Q_{1}$.

Let $\Gamma (Q)$ denote the graph underlying $Q$. We say $Q$ is
\emph{star-shaped} if there exists $m \in \mathbb{N}$ and
$n_{1}, \dots , n_{m} \in \mathbb{N}$ such that $\Gamma (Q)$ is a quotient
of the $m$ type A graphs (see Fig.~\ref{fig:Dynkin_diagrams})
$\sqcup _{i \in \{1, 2, \dots m \}} \Gamma (A_{n_{i}}) $ formed by identifying
the $1 \in (A_{n_{1}})_{0}$ with $1 \in (A_{n_{i}})_{0}$ for all
$i \in \{ 1, \dots , m \}$. The subquivers $A_{n_{i}} \subset Q$ are called
\emph{arms}, which for consistency we always orient towards the central
vertex. For a general star-shaped quiver $Q$ we write $v$ for the central
vertex, while we set $v = 3$ when working with a specific Dynkin quiver
as in Fig.~\ref{fig:Dynkin_diagrams}.

\begin{exam}
\label{exmp2.1}
Let $D_{5}$ be the quiver (depicted below) with vertex set
$\{1, 2, 3, 4, 5 \}$, arrow set $\{ a, b, c, d \}$, and with e.g.
$s(a) =1$, $t(a) = 3$. $D_{5}$ is star-shaped with three arms
$1 \overset{a}{\rightarrow} 3$, $2 \overset{b}{\rightarrow} 3$ and
$5 \overset{d}{\rightarrow} 4 \overset{c}{\rightarrow} 3$. The opposite
and doubled quivers are
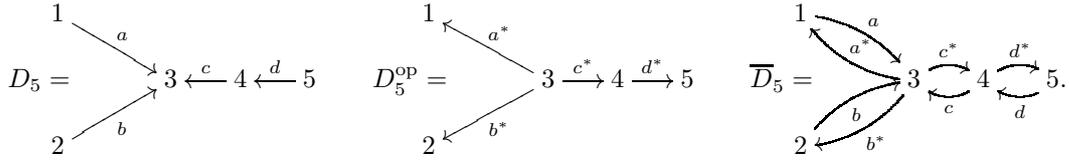
\begin{equation*}
\xymatrix@R=1em@C=1.5em{ 1 \ar[rd]^a & & & & 1 & & & & 1 \ar@/^.5pc/[rd]^a & & & & \\
D_{5} = \hspace{.5cm} & 3 & 4 \ar[l]_{c} & \ar [l]_{d} 5 & D_{5}^{\text{op}} = \hspace{.5cm} & 3 \ar[r]^{c^{*}} \ar[lu]_{a^{*}} \ar[ld]^{b^{*}} & 4 \ar[r]^{d^{*}} & 5 & \overline{D}_{5} = \hspace{.5cm} & 3 \ar@/^.5pc/[r]^{c^{*}} \ar@/^.5pc/[lu]_{a^{*}} \ar@/^.5pc/[ld]^{b^{*}} & 4 \ar@/^.5pc/[r]^{d^{*}} \ar@/^.5pc/[l]^{c} & \ar @/^.5pc/[l]^{d} 5. \\
2 \ar[ru]_b & & & & 2 & & & & 2 \ar@/^.5pc/[ru]_{b} & & & &}
\end{equation*}
\end{exam}

For any commutative ring $R$, let $R Q$ denote the path algebra (i.e. the
algebra of $R$-linear combinations of paths) of the quiver. Multiplication
in $R Q$ is given by concatenation (which we write from left to right)
where possible, and zero otherwise. The path algebra includes length-zero
(or lazy) paths at the vertices which we denote $e_{i}$ for
$i \in Q_{0}$.
We will call a path $p = a_{1} \cdots a_{n}$ in $RQ$ a \emph{cycle} if
$s(a_{1}) = t(a_{n})$ and further call it a \emph{loop} if $n=1$.

\begin{defn}
\label{defn2.2}
The (additive) \emph{preprojective algebra} of a quiver $Q$, denoted
$\Pi _{R}(Q)$, is the quotient of the path algebra of the doubled quiver
$R \overline{Q}$ modulo the two-sided ideal generated by the relation
\begin{equation*}
r_{\text{add}} := \sum _{a \in Q_{1}} aa^{*} - a^{*} a .
\end{equation*}
\end{defn}

\begin{defn}
\label{defn2.3}
\cite[Definition 1.2]{CBS06} The
\emph{multiplicative preprojective algebra} of a quiver $Q$, denoted
$\Lambda _{R}(Q)$, is the quotient of the localized path algebra
$L := R \overline{Q} [ (1+a^{*}a)^{-1} ]_{a \in Q_{1}}$ modulo the two-sided
ideal generated by the relation
\begin{equation*}
r_{\text{mult}} := \prod _{a \in Q_{1}} (1+ aa^{*})(1+a^{*}a)^{-1} - 1.
\end{equation*}
\end{defn}

As $\Pi _{R}(Q)$ is graded by path length we write
$\Pi _{R}^{\geq i}(Q)$, $\Pi _{R}^{\leq i}(Q)$, and
$\Pi _{R}^{+}(Q)$ for the $R$-subspace generated by paths of length at
least $i$, paths of length at most $i$, and positive length paths respectively.

\begin{rem}
\label{rem:_indep_of_ordering}
Note the product is taken with respect to some choice of ordering of the
arrows in $\overline{Q}_{1}$. Crawley-Boevey and Shaw prove that the isomorphism
class of $\Lambda _{R}(Q)$ is independent of the choice of ordering in
\cite[Theorem 1.4]{CBS06}. By conjugation, it is clear only the cyclic
order is relevant and hence for $Q$ star-shaped one only needs a cyclic
order at the central vertex. In this paper we label the arrows leaving
the central vertex as $a^{*}, b^{*}$, and $c^{*}$ and always use the alphabetical
order $a^{*} < b^{*} < c^{*}$.
\end{rem}

\begin{rem}
\label{rem:_opposite}
Crawley-Boevey and Shaw read the concatenated path
$1 \overset{a}{\rightarrow} 2 \overset{b}{\rightarrow} 3$ as $ba$ (from
right to left) while our convention is to read this path as $ab$ from left
to right. By \cite[Theorem 1.4]{CBS06} the isomorphism class of
$\Lambda _{R}(Q)$ is independent of the orientation of $Q$, and hence
$\Lambda _{R}(Q) \cong \Lambda _{R}(Q^{\text{op}}) = \Lambda _{R}(Q)^{
\text{op}}$. For the usual preprojective algebra
$\Pi _{R}(Q) \cong \Pi _{R}(Q)^{\text{op}}$ via the map defined on generators
by $e_{i} \mapsto e_{i}$ and $a \mapsto a^{*}$, $a^{*} \mapsto -a$ for
$a \in Q_{1}$.
\end{rem}

\begin{rem}
\label{rem:_relations_agree_mod_4}
The relation $r_{\text{mult}}$ can be realized as $r_{\text{add}}$ modulo
paths of length 4, as we now explain. Since the isomorphism class of the
algebra is independent of the choice of order on the arrows in
$\overline{Q}_{1}$, we choose an order with $a < b^{*}$ for all
$a, b \in Q_{1}$. So
\begin{equation*}
r_{\text{mult}} := \prod _{a \in Q_{1}} (1+aa^{*}) \prod _{a \in Q_{1}}
(1+a^{*}a)^{-1} -1.
\end{equation*}
Since $(1+ a^{*}a)^{-1} = 1 - a^{*} a \ + $ (higher order terms) we conclude
that
$r_{\text{mult}} = r_{\text{add}} + \text{(higher order terms)}$. Note if
$Q$ is star-shaped then there are only finitely many higher order terms
by Corollary~\ref{cor:_inverses}.
\end{rem}

We present the following unifying setting for studying these algebras.

\begin{defn}
\label{defn2.7}
Let $w \in Q_{0}$. The \emph{partial} preprojective algebra of
$(Q, w)$, denoted $\Pi _{R}(Q, w)$, is the quotient
$R \overline{Q} / ( (1-e_{w}) r_{\text{add}})$. Similarly, the
\emph{partial} multiplicative preprojective algebra of $(Q, w)$, denoted
$\Lambda _{R}(Q, w)$, is the quotient
$L / ((1-e_{w}) r_{\text{mult}})$.
\end{defn}

In words, this is the usual (multiplicative) preprojective algebra of
$Q$ except we do not enforce the relation at the vertex $w$.

We will show $\Pi _{R}(Q, v) = \Lambda _{R}(Q, v)$ for $Q$ star-shaped
with central vertex $v$, but first we need the following elementary result.

\begin{lem}
\label{lem2.8}
For all $i \in \{ 0, \dots , n-1 \}$,
$(a_{i} a_{i}^{*})^{i} = 0 \in \Pi _{R}(A_{n}, n)$.
\end{lem}

\begin{proof}
When $i=1$ the preprojective relation at vertex $1$ is
$a_{1} a_{1}^{*} =0$. When $i= 2$ the preprojective relation at vertex
$2$, $a_{2} a_{2}^{*} - a_{1}^{*} a_{1}$, implies
\begin{equation*}
(a_{2} a_{2}^{*})^{2} = (-a_{1}^{*} a_{1})^{2} = a_{1}^{*} ( a_{1} a_{1}^{*})
a_{1} = 0.
\end{equation*}
Continuing in this way, one can pull any path $(a_{i} a_{i})^{i}$ towards
vertex $1$ using the preprojective relation at the vertices
$2, 3, \dots , i$, at which point the relation at vertex 1 implies the
path is zero.
\end{proof}

\begin{cor}
\label{cor:_inverses}
For $Q$ star-shaped with central vertex $v$ and $a \in Q_{1}$, the cycle
$a^{*}a \in \Pi _{R}(Q, v)$ is nilpotent.
\end{cor}

\begin{proof}
Since $Q$ is star-shaped each $a^{*}a$ is contained in some arm
$A_{n} \subset Q$. Hence
$(a^{*}a)^{n} = 0 \in \Pi _{R}(A_{n}, n) \subset \Pi _{R}(Q, v)$.
\end{proof}

\begin{prop}
\label{prop:_partial_equal}
For $Q$ star-shaped with central vertex $v$,
$\Pi _{R}(Q, v) = \Lambda _{R}(Q, v)$.
\end{prop}

\begin{proof}
Let $i \in Q_{0}$ be valence two with incoming arrow $a$ and outgoing arrow
$b$. Then $e_{i} r_{\text{mult}}$ is
\begin{equation*}
(1+ a^{*}a)^{-1}(1+bb^{*}) = 1 \quad \implies \quad bb^{*} - a^{*} a =0
\end{equation*}
so $e_{i} r_{\text{mult}} = e_{i} r_{\text{add}}$ for all $i \neq v$ (where
the endpoints are the case the incoming arrow $a = 0$.)

By Corollary~\ref{cor:_inverses}, $a^{*}a$ is nilpotent in
$\Pi _{R}(Q, v)$ and hence
\begin{equation*}
(1+ a^{*}a)^{-1} = 1-a^{*}a + (a^{*}a)^{2} - \cdots \pm (a^{*}a)^{|Q_{0}|}
\in \Pi _{R}(Q, v).
\end{equation*}
We conclude that enforcing the relations renders the localization unnecessary
and hence
\begin{equation*}
\Lambda _{R}(Q) := L/( (1-e_{v})r_{\text{mult}}) = R \overline{Q}/( (1-e_{v})r_{
\text{mult}}) = R \overline{Q}/( (1-e_{v}) r_{\text{add}}) =: \Pi _{R}(Q).\qedhere
\end{equation*}
\end{proof}

Two immediate consequences are that
$\Lambda _{R}(A_{n}) = \Pi _{R}(A_{n})$ for any $n \in \mathbb{N}$ and
$\Lambda _{R}(Q)$ is a quotient of $R \overline{Q}$ for $Q$ star-shaped.

\section{Correcting algebra isomorphisms}
\label{sec:_correct_iso}

Throughout this section, let $k$ be a field and let $Q$ be a connected
quiver without loops and with at most one arrow with a given source and
target (i.e., without double arrows). Write $\Pi (Q) := \Pi _{k}(Q)$ and
$\Lambda (Q) := \Lambda _{k}(Q)$. We show that any $k$-algebra isomorphism
$\Lambda (Q) \rightarrow \Pi (Q)$ can be corrected to an isomorphism that
is the identity on vertices and takes arrows to themselves plus higher
order terms. For the reader interested solely in $kQ_{0}$-module isomorphisms
Proposition~\ref{prop:_correct_iso} is automatic and can be skipped.

We prove this result in the following more general setting. Let
$J, J'$ be ideals in $kQ$ and let $A := kQ/J$ and $B := kQ/J'$ be quotient
algebras. The \emph{Peirce decomposition} of $A$ is
$A \cong \oplus _{i, j \in Q_{0}} e_{i} A e_{j}$. Let $kQ_{\geq 2}$ be
the vector space spanned by paths of length at least two and let
$B_{\geq 2}$ denote its image in $B$ under the projection map.
%
\begin{defn}
\label{defn3.1}
Let $\phi : A \rightarrow B$ be a $k$-algebra homomorphism. Define
$\phi $ to be \emph{vertex-preserving} if $\phi (e_{i}) = e_{i}$ for all
$i \in Q_{0}$. Define $\phi $ to be \emph{decomposition-preserving} if
$\phi ( e_{i} A e_{j}) \subset e_{i} B e_{j}$ for all
$i, j \in Q_{0}$, i.e., it preserves the Peirce decomposition. Define
$\phi $ to be \emph{triangular} if
$\phi (a) \equiv c_{a} a \ (\text{mod } B_{\geq 2})$ where
$c_{a} \in k$, and $\phi $ to be further \emph{unitriangular} if $c_{a} =1$ for
all $a \in Q_{1}$.
\end{defn}

\begin{lem}
\label{lem:_triangular}
If $\phi $ is vertex-preserving then it is decomposition-preserving and
hence is triangular if $Q$ has no double arrows.
\end{lem}

\begin{proof}
Let $\phi $ be vertex-preserving, so
$\phi ( e_{i} A e_{j}) = e_{i} \phi ( A ) e_{j} \subset e_{i} B e_{j}$
and $\phi $ is decomposition-preserving. For each arrow
$a \in Q_{1}$, the general expression for its image under $\phi $ is
\begin{equation*}
\phi (a) \equiv \sum _{i \in Q_{0}} c_{i} e_{i} + \sum _{b \in Q_{1}} c_{b}
b \  (\text{mod } B_{\geq 2}).
\end{equation*}
Since $\phi $ is decomposition-preserving and
$e_{s(a)} Q_{1} e_{t(a)} = \{ a, 0 \}$,  $\phi (a) \equiv c_{a} a \ (
\text{mod } B_{\geq 2})$ as desired.
\end{proof}

\begin{defn}
\label{defn3.3}
An ideal $J \subset k Q$ is \emph{$Q$-symmetric} if any automorphism of
$Q$ preserves $J$.
\end{defn}

Any automorphism of $Q$ defines a unique algebra automorphism of
$k Q$. Furthermore, if $J$ is $Q$-symmetric, then this algebra map descends
to the quotient $k Q /J$. In other words this condition ensures that
$\text{Aut}(Q) \subset \text{Aut}(A)$. Examples of $Q$-symmetric ideals include
the ideal of definition in the multiplicative and additive preprojective
algebras.
%
\begin{prop}
\label{prop:_correct_iso}
Let $Q$ be a connected quiver without loops and without double arrows.
Define $A := k Q/J$, $B := k Q/J'$, where $J' \subset B_{\geq 2}$ and is
$Q$-symmetric and $J \subset A_{\geq 2}$. If
$\phi : A \rightarrow B$ is a $k$-algebra isomorphism then there exists
a vertex-preserving, (and hence triangular) $k$-algebra isomorphism
$\varphi : A \rightarrow B$ such that $\varphi = \psi \circ \phi $ for
some $\psi \in \text{Aut}(B)$.
\end{prop}

\begin{rem}
\label{rem3.5}
If $A = \Lambda (D_{4})$ and $B = \Pi (D_{4})$ then Proposition~\ref{prop:_correct_iso} and the proof technique below is due to Shaw in
\cite[Lemma 5.2.1]{Shaw05}. Our first two steps below mimic his, after
which we deviate and generalize his argument to no longer require specific
knowledge of $A$ or $B$, beyond our assumptions.
\end{rem}

\begin{proof}
The strategy of the proof is to fix the isomorphism $\phi $ and postcompose
with automorphisms of $B$ until it is vertex-preserving. By Lemma~\ref{lem:_triangular} it is then triangular.

\bigskip

\noindent \underline{Step 1}: Correct $\phi $ to a map $\varphi $ taking
$kQ_{0}$ to $kQ_{0}$.

Fix $\phi : A \rightarrow B$ and notice $\phi (kQ_{0})$ is a separable
subalgebra of $B$ and hence by Wedderburn--Malcev \cite{Malcev42} conjugate
to $kQ_{0}$ in $B$. Postcomposing by this conjugation gives a $k$-algebra
map $\varphi :A \rightarrow B$ satisfying
$\varphi (kQ_{0}) = kQ_{0}$.

\bigskip

\noindent \underline{Step 2}: Observe that $\varphi $ restricted to $Q_{0}$ is a
permutation of $Q_{0}$.

Write $\varphi (e_{i} ):= \sum c_{j}^{i} e_{j}$ and note that the matrix
$C := [c_{j}^{i}]$ is a permutation matrix. Indeed, since
$e_{i} e_{\ell} = \delta _{i,\ell} e_{i}$ where $\delta _{i, \ell}$ is
one if $i = \ell $ and zero otherwise, we can compute
$\varphi (e_{i} e_{\ell})$ two ways:
\begin{align*}
\varphi (e_{i} e_{\ell}) &= \varphi (e_{i}) \varphi (e_{\ell}) =
\left ( \sum _{j} c_{j}^{i} e_{j} \right ) \left ( \sum _{m} c_{m}^{
\ell} e_{m} \right ) = \sum _{j=m} c_{j}^{i} c_{j}^{\ell} e_{j}
\\
\varphi (e_{i} e_{\ell}) &= \varphi ( \delta _{i, \ell} e_{i}) =
\delta _{i, \ell} \sum _{j} c_{j}^{i} e_{j}.
\end{align*}
Equating both expressions and multiplying by $e_{j}$ (i.e., extracting
the coefficient in front of $e_{j}$) gives
\begin{equation*}
c_{j}^{i} c_{j}^{\ell} =
\begin{cases}
0 & \text{ if } i \neq \ell
\\
c_{j}^{i} & \text{ if } i = \ell .
\end{cases}
\end{equation*}
In particular, $(c_{j}^{i})^{2} = c_{j}^{i}$ hence $c_{j}^{i} = 0$ or
$1$. And each row and column of $C$ has a single non-zero entry. We conclude
that $C$ is a permutation matrix and write
$\sigma := \varphi \mid _{Q_{0}}$ for this permutation of $Q_{0}$.

\bigskip

\noindent \underline{Step 3}: $\sigma $ is induced from an automorphism of $Q$.

Note that the assumption $J \subset A_{\geq 2}$ implies we have a decomposition
\begin{equation*}
A = A_{0} \oplus A_{1} \oplus A_{\geq 2} = kQ_{0} \oplus kQ_{1}
\oplus A_{\geq 2}
\end{equation*}
where $A_{\geq 2} \cdot A_{\geq 2} \subset A_{\geq 2}$, and similarly for
$B$. By Step 1, $\varphi (A_{0}) = B_{0}$ and we will next establish
$\varphi (A_{\geq 2}) \subset B_{\geq 2}$ and hence for $\varphi $ to be
surjective on arrows we need
$\varphi (A_{1}) = B_{1} \text{ (mod } B_{\geq 2})$.

Let $a \in Q_{1}$ be an arrow from $i$ to $j$, viewed as a non-zero element
in $A$. The injectivity of $\varphi $ implies
\begin{equation*}
0 \neq \varphi (a) = \varphi (e_{i} a e_{j}) = e_{\sigma (i)}
\varphi (a) e_{\sigma (j)}.
\end{equation*}
By assumption on $Q$, there is a at most one arrow $b$ from
$\sigma (i)$ to $\sigma (j)$. We claim such a $b$ must exist for
$\varphi $ to be surjective. Indeed since $Q$ has no loops,
$i \neq j$, and hence $\varphi (a) \subset B_{\geq 1}$ for all arrows
$a \in Q_{1}$. So for two arrows $a'$ and $a''$,
$\varphi ( a' a'') \subset B_{\geq 2}$ and hence
$\varphi (A_{\geq 2}) \subset B_{\geq 2}$. So $kQ_{1} \subset B$ must be
in the image of $kQ_{1} \subset A$ modulo $B_{\geq 2}$. Since
$\varphi $ is a permutation on the vertices, it follows that it permutes
(and possibly scales) the arrows modulo $B_{\geq 2}$.

Hence the assignment $i \mapsto \sigma (i)$ and $a \mapsto b$ is an automorphism
of $Q$. It follows that $\sigma $ is induced from such an automorphism.

\bigskip

\noindent \underline{Step 4}: Correct $\varphi $ to a vertex-preserving map.

By Step 3, $\sigma $ comes from an automorphism of $Q$, which one can extend
uniquely to an algebra map $\tilde{\sigma}: k Q \rightarrow k Q$. Since
$J'$ is $Q$-symmetric by assumption, $\tilde{\sigma}(J') = J'$ and hence
$\tilde{\sigma}$ descends to an automorphism
$\psi : B \rightarrow B$. The map
$\psi ^{-1} \circ \varphi : A \rightarrow B$ is a vertex-preserving
$k$-algebra isomorphism.
\end{proof}

\begin{rem}
\label{rem3.6}
The condition that $J'$ is $Q$-symmetric is stronger than required. One
only needs that the permutation of the vertices $\sigma ^{-1}$ extends
to \emph{some} graph automorphism that set-wise fixes $J'$, and hence it
suffices that the set of graph automorphisms preserving $J'$ exhaust all
permutations of the vertices among $\text{Aut}(Q)$.

This condition cannot be removed entirely as we now demonstrate. Let
$Q=A_{3}$ with vertex set $\{ 0, 1, 2 \}$ arrow set $\{a, b\}$ such that
$s(a) = 0$, $t(a) = 1 = t(b)$, $s(b) = 2$. Define
$A = k\overline{Q}/J$ where $J =(aa^{*})$ and $B= k \overline{Q}/J'$ where
$J' =(bb^{*})$. Notice the symmetry in $\overline{A}_{3}$ (reflecting over
the central vertex 1) induces the $k$-algebra isomorphism
\begin{equation*}
\phi : k \overline{Q} \rightarrow k \overline{Q}
\qquad\quad
\text{such that}
\qquad\quad
\phi (e_{i}) = e_{2-i}, \ \phi (a) = b, \ \phi (a^{*}) = b^{*}, \
\phi (b) = a, \  \phi (b^{*}) = a^{*}.
\end{equation*}
Since $\phi (aa^{*}) = bb^{*}$ this map descends to an isomorphism
$\psi : A \rightarrow B$ sending $e_{0} \mapsto e_{2}$. But $\psi $ cannot
be corrected to send $e_{0} \mapsto e_{0}$ while still sending the relation
$aa^{*} = e_{0} aa^{*} e_{0} \in J$ to
$bb^{*} = e_{2} b^{*}b e_{2} \in J'$.
\end{rem}

We apply Proposition~\ref{prop:_correct_iso} in the setting of additive
and multiplicative preprojective algebras.

\begin{cor}
\label{cor:_correct_to_vertex-preserving}
Let $\overline{Q}$ be as in Proposition~\ref{prop:_correct_iso}. Let
$\phi : \Lambda (Q) \rightarrow \Pi (Q)$ be an isomorphism of $k$-algebras.
Then there exists a vertex-preserving $k$-algebra isomorphism
$\varphi : \Lambda (Q) \rightarrow \Pi (Q)$.
\end{cor}

\begin{proof}
It suffices to show that we can take $A= \Lambda (Q)$ and
$B=\Pi (Q)$ in Proposition~\ref{prop:_correct_iso}. By
\cite[Theorem 2.2.8]{Shaw05}, since $Q$ does not have loops,
$\Lambda (Q)$ can be realized as a quotient of the path algebra,
$k\overline{Q}$, \emph{without} the need to localize. (In the star-shaped
case this result is easier and given in Corollary~\ref{cor:_inverses}.)
The relations for both the additive and multiplicative preprojective algebras
lie in $k\overline{Q}_{\geq 2}$ and respect the symmetry of the quiver
$Q$.
\end{proof}

\begin{lem}
\label{lem:_unitriangular}
If $\varphi : \Lambda (Q) \rightarrow \Pi (Q)$ is a triangular isomorphism
then there exists a unitriangular isomorphism $\theta $ and an automorphism
$\psi $ of $\Pi (Q)$ such that $\theta = \psi \circ \varphi $.
\end{lem}

\begin{proof}
Let $\varphi : \Lambda (Q) \rightarrow \Pi (Q)$ be triangular, meaning
for each $a \in Q_{1}$,
$\varphi (a) = c_{a} a + (\text{higher order terms})$ for some
$c_{a} \in k$. Note $\varphi $ surjective implies $c_{a} \neq 0$. For
$\varphi $ to be well-defined the quadratic piece of
$\varphi (r_{\text{mult}})$ must lie in the ideal generated by
$r_{\text{add}}$. As
$r_{\text{mult}} = r_{\text{add}} + \text{(higher order terms)}$, the quadratic
piece is
\begin{equation*}
\varphi (r_{\text{mult}})_{\leq 2} = \sum _{a \in Q_{1}} c_{a} c_{a^{*}}
aa^{*} - c_{a^{*}} c_{a} a^{*} a.
\end{equation*}
Hence, for any two arrows $a$ and $b$ with the same source or target, we
conclude that the coefficients $c_{a} c_{a^{*}} = c_{b}c_{b^{*}}$. Therefore,
postcomposing by the automorphism
\begin{equation*}
a \mapsto a/c_{a} \text{ for } a \in \overline{Q}_{1} \quad e_{i}
\mapsto e_{i} \text{ for } i \in Q_{0}
\end{equation*}
of $\Pi (Q)$ further rigidifies $\varphi $ to a unitriangular isomorphism
$\theta $, i.e., $\theta (a) = a + (\text{higher order terms})$.
\end{proof}

Combining Corollary~\ref{cor:_correct_to_vertex-preserving}, Lemma~\ref{lem:_triangular}, and Lemma~\ref{lem:_unitriangular} we get the following:
%
\begin{cor}
\label{cor:_correct_iso}
Let $\overline{Q}$ be a connected quiver without loops and without double
arrows. For any $k$-algebra isomorphism
$\phi : \Lambda (Q) \rightarrow \Pi (Q)$ there exists a $k$-algebra isomorphism
$\varphi : \Lambda (Q) \rightarrow \Pi (Q)$ that preserves the
$kQ_{0}$-module structure and sends each arrow to itself plus higher order
terms.
\end{cor}

\begin{rem}
\label{rem:_normal_form}
We will need the stronger condition that the map $\varphi $ sends certain
\emph{paths} to themselves plus higher order terms. If the path is in so-called
normal form (with respect to some choice of ordering) when viewed in both
$\Lambda (Q)$ and $\Pi (Q)$ then this follows from the corresponding result
on arrows. Indeed this is the case for the relevant paths in Section~\ref{sec:_HH_0} (i.e., $\alpha \beta $, $\beta \alpha \beta $, and
$\beta \alpha \beta \alpha \beta $ with $Q$ ADE Dynkin) using the lexicographical
ordering with $b^{*} > a^{*} > b > a$ for any pair of arrows
$(a, b)$ with $a < b$ alphabetically.
\end{rem}

\begin{defn}
\label{def:_corrected}
We call a function $\overline{Q}_1 \rightarrow \Pi (Q)$
\emph{corrected} if it sends each arrow
to itself plus higher order terms. We call a $k$-algebra isomorphism between
$\Lambda (Q)$ and $\Pi (Q)$ \emph{corrected} if it preserves the
$kQ_{0}$-module structure and sends each arrow to itself plus higher order
terms.
\end{defn}

Notice that Corollary~\ref{cor:_correct_iso} makes both finding isomorphisms
and proving their non-existence easier, since we only need to consider
corrected isomorphisms. In the next section we find explicit obstructions
to an isomorphism using a discrepancy in zeroth Hochschild homology.

\section{Zeroth Hochschild homology}
\label{sec:_HH_0}

Let $Q$ be an ADE Dynkin quiver and let $R$ be a commutative ring with
unit $1_{R}$. Let $p$ be a bad prime for $Q$, see Definition~\ref{def:_bad_prime}, viewed as an element of $R$ via
$p \cdot 1_{R}$. In this section, we prove
$\Lambda _{R}(Q) \ncong \Pi _{R}(Q)$ if $p$ is not invertible in
$R$ by showing that
\begin{equation*}
\text{HH}_{0}(\Lambda _{R}(Q)) = RQ_{0} \ncong \text{HH}_{0}(\Pi _{R}(Q)).
\end{equation*}

In this section, all algebras are over the integers unless denoted otherwise
with a subscript.

\begin{defn}
\label{defn4.1}
For an algebra $A$ the zeroth Hochschild homology is
$\text{HH}_{0}(A) := A/[A, A] =: A_{\text{cyc}}$ where $[A, A]$ is the
$\mathbb{Z}$-linear span of the set of commutators in $A$.
\end{defn}

Schedler computed $\Pi (Q)_{\text{cyc}}$ with $Q$ ADE Dynkin in
\cite[Theorem 13.1.1]{Schedler16} and found $p$ torsion for each $p$ bad
for $Q$. Here we compute $\Lambda (Q)_{\text{cyc}}$ to be torsion free for
$Q$ ADE Dynkin. Hence over the integers (or more generally over any commutative
ring $R$ with $p$ not invertible) this gives an obstruction to the existence
of an isomorphism $\Lambda (Q) \rightarrow \Pi (Q)$.

While the above argument requires computing all of
$\Lambda (Q)_{\text{cyc}}$, one can refined this obstruction to a single
class $[c] = 0 \in \text{HH}_{0}(\Lambda _{R}(Q))$ whose image is non-zero
in $\text{HH}_{0}(\Pi _{R}(Q))$, as we now explain.

If $p$ is not invertible in $R$ then choose a maximal ideal
$\mathfrak{m} \subset R$ containing $p$, to obtain the field
$k := R/ \mathfrak{m}$ of characteristic $p$. Any isomorphism
$\Lambda _{R}(Q) \rightarrow \Pi _{R}(Q)$, descends modulo
$\mathfrak{m}$ to an isomorphism
$\Lambda _{k}(Q) \rightarrow \Pi _{k}(Q)$ of $k$-algebras. By Corollary~\ref{cor:_correct_iso}, we can correct this map to one taking each path
(in normal form) to itself plus higher order terms. The corrected map induces
an isomorphism
$\Lambda _{k}(Q)_{\text{cyc}} \cong \Pi _{k}(Q)_{\text{cyc}}$, that takes
$[c] = 0 \in \Lambda _{k}(Q)_{\text{cyc}}$ to
$[c] = 0 \in \Pi _{k}(Q)_{\text{cyc}}$, utilizing the grading in
$\Pi _{k}(Q)$. Therefore, to show
$\Lambda _{R}(Q) \ncong \Pi _{R}(Q)$ it suffices to find a single cycle
$c$, in normal form, with $[c] = 0 \in \Lambda _{k}(Q)_{\text{cyc}}$ and
$[c] \neq 0 \in \Pi _{k}(Q)_{\text{cyc}}$.

If $p=2$, $p=3$, and $p=5$ respectively, then such a length $2p$ cycle,
$c$, is given by $\alpha \beta , \beta \alpha \beta $, and
$\beta \alpha \beta \alpha \beta $ from Lemma~\ref{lem:_length_4},
\ref{lem:_length_6}, and \ref{lem:_length_10}.%

Let $Q$ be star-shaped with central vertex $0$. We begin with a few elementary
results about $\Pi (Q, 0)_{\text{cyc}}$ for the \emph{partial} preprojective
algebra with $Q$ as this provides a unified setting (1) for all ADE Dynkin
quivers and (2) for the additive and multiplicative preprojective algebras
by Proposition~\ref{prop:_partial_equal}.

Let $A \subset Q$ be an arm
\begin{equation*}
A =0 \overset{c_{1}}{\longleftarrow} 1
\overset{c_{2}}{\longleftarrow} 2 \overset{c_{3}}{\longleftarrow} 3
\overset{c_{4}}{\longleftarrow}\cdots \overset{c_{n}}{\longleftarrow} n.
\end{equation*}
%
\begin{lem}
\label{lem:_vanishing_cycles}
$\Pi (A, 0)_{\text{cyc}} = \mathbb{Z}A_{0}$, so every positive length path
is zero.
\end{lem}

\begin{proof}
The preprojective relation at vertex $i$ is
$c_{i} c_{i}^{*} - c_{i+1}^{*} c_{i+1} = 0$ for $i \neq n$ and at vertex
$n$ is $c_{n} c_{n}^{*} = 0$. Let $c$ be a positive length cycle. If
$c$ has a subpath of the form $c_{i} c_{i}^{*}$, for any $i$, then use
the preprojective relations to replace it with
$ c_{i+1}^{*} c_{i+1}$ if $i<n$ or 0 if $i=n$. If $c$ has no subpath of
this form then it necessarily begins with a dual arrow $c_{l}^{*}$ and
ends with $c_{l}$ for some $l$. Hence $c$ can be cyclically reordered to
begin $c_{l} c_{l}^{*}$. Repeating this process of pulling paths to the
right and reordering paths that cannot be pulled right eventually terminates
(as $c$ is finite length and $n$ is finite) with zero.
\end{proof}

\begin{cor}
\label{cor:HH0_An_R}
$\Pi_R(A)_{\text{cyc}} = R A_{0}$ for any commutative ring R. 
\end{cor}

\begin{proof}
The quotient map $\Pi(A, 0) \rightarrow \Pi(A)$ is an isomorphism in degree zero and so by Lemma \ref{lem:_vanishing_cycles}, identifies $\Pi(A, 0)_{\text{cyc}} = \mathbb{Z}A_{0} = \Pi(A)_{\text{cyc}}$. Since the result holds over $\mathbb{Z}$, it holds over any commutative ring $R$.
\end{proof}

\begin{lem}
\label{lem:_spanning_set_general}
$\Pi (Q, 0)_{\text{cyc}}$ has spanning set
$\mathbb{Z}Q_{0} \oplus e_{0} \Pi ^{\geq 4}(Q, 0) e_{0}$, where
$\Pi ^{\geq 4}(Q, 0)$ is the subspace of length at least 4 paths. Since
$\Pi (Q, 0)_{\text{cyc}} = \Lambda (Q, 0)_{\text{cyc}}$ it follows that this
is a spanning set for the quotient $\Lambda (Q)$.
\end{lem}

\begin{proof}
Let $c$ be a positive length cycle in $\Pi (Q, 0)$. If $c$ does not visit
$0$ twice then $c$ is contained in some arm $A \subset Q$ and hence, by
Lemma~\ref{lem:_vanishing_cycles},
$c = 0 \in \Pi (A, 0)_{\text{cyc}} \subset \Pi (Q, 0)_{\text{cyc}}$. Otherwise
$c$ has length at least 4 and visits $0$, so up to cyclic ordering we can
rewrite $c$ as a cycle starting and ending at $0$.
\end{proof}

\noindent
\underline{Notation:} For $Q = D_{n}, E_{m}$ for
$n \geq 4, m = 6, 7, 8$ we orient all arrows towards the central vertex
$3$, as in Fig.~\ref{fig:Dynkin_diagrams}. We label the three arrows
with target $3$ by $a, b, c$ where $a$ is in the shortest arm and
$c$ is in the longest arm. We write $a^{*}, b^{*}$, and $c^{*}$ respectively
for the dual arrows and use the Greek letters
$\alpha := a^{*} a$, $\beta := b^{*} b$, $\gamma := c^{*}c$ for the length
two cycles at vertex $3$. Note that in this language the multiplicative
preprojective relation at vertex $3$ is
$(1+ \alpha )^{-1}(1+\beta )^{-1}(1+\gamma )^{-1} = 1$.

\begin{rem}
\label{rem:_linearly_indep}
For the purpose of computing $\Lambda (Q)_{\text{cyc}}$ we will begin with
the spanning set in Lemma~\ref{lem:_spanning_set_general} and whittle it
down to $\mathbb{Z}Q_{0}$. To demonstrate the torsion in $\Pi (Q)_{\text{cyc}}$ is nonzero,
we use \cite[Proposition 11.3.2 (iii)]{Schedler16}, which is not hard to
prove from what we have already established. Namely, if $Q$ is the star-shaped
quiver with three arms of length $p$, $q$, and $r$ then
%
\begin{equation}
\label{eq:_HH_0_description}
\Pi (Q)_{\text{cyc}} := \mathbb{Z}Q_{0} \oplus
\frac{A}{ [A, A] \oplus P}
\qquad\quad
\text{with } A := e_{3} \Pi (Q) e_{3} \cong
\frac{\mathbb{Z}\langle \alpha , \beta , \gamma \rangle }{(\alpha + \beta + \gamma , \alpha ^{p+1}, \beta ^{q+1}, \gamma ^{r+1})}
\end{equation}
where $P$ is the $\mathbb{Z}$-span of
$\{ \alpha ^{i}, \beta ^{i}, \gamma ^{i} \}_{i \in \mathbb{N}}$, i.e.,
the powers of length two cycles based at the central vertex.

Moreover, since $\Pi (Q)_{\text{cyc}}$ is graded, to show a cycle $c$ of
length $d$ is non-zero it suffices to consider the length $d$ piece of
$P$. Notice for $Q=D_{4}$ the length $4$ piece of $P$ is
$\mathbb{Z}\{ \gamma ^{2} \}$. Similarly for $Q=E_{6}$ the length
$6$ piece of $P$ is $\mathbb{Z}\{ \gamma ^{3} \}$ and for $Q=E_{8}$ the
length $10$ piece of $P$ is $\mathbb{Z}\{ \gamma ^{5} \}$. So the identifications
\begin{align*}
\gamma ^{2} &= (\alpha + \beta )^{2} = \alpha \beta + \beta \alpha = 2
\alpha \beta
\\
-\gamma ^{3} &= (\alpha + \beta )^{3} = \alpha \beta \alpha + \alpha
\beta ^{2} + \beta \alpha \beta + \beta ^{2} \alpha = 3 \beta \alpha
\beta
\\
-\gamma ^{5} &= (\alpha + \beta )^{5} = \alpha \beta \alpha \beta
\alpha + \alpha \beta \alpha \beta ^{2} + \alpha \beta ^{2} \alpha
\beta + \beta \alpha \beta \alpha \beta + \beta \alpha \beta ^{2}
\alpha + \beta ^{2} \alpha \beta ^{2} = 5 \beta \alpha \beta \alpha
\beta
\end{align*}
imply $2 \alpha \beta $, $3 \beta \alpha \beta $, and
$5 \beta \alpha \beta \alpha \beta = 0 \in \Pi (Q)_{\text{cyc}}$ for
$Q$ ADE Dynkin. Since these relations exhaust $P$,
$\alpha \beta \neq 0 \in \Pi (D_{4})_{\text{cyc}}$,
$\beta \alpha \beta \neq 0 \in \Pi (E_{6})_{\text{cyc}}$ and
$\beta \alpha \beta \alpha \beta \neq 0 \in \Pi (E_{8})_{\text{cyc}}$.
\end{rem}

\begin{prop}
\label{prop:_HH_0_D_n}
$\Lambda (D_{n})_{\text{cyc}} = \mathbb{Z}(D_{n})_{0}$.
\end{prop}

\begin{proof}
By Lemma~\ref{lem:_spanning_set_general},
$\Lambda (D_{n})_{\text{cyc}}$ has a spanning set given by idempotents at
the vertices and words in $\alpha , \beta , \gamma $. Since
$\alpha ^{2} = 0 = \beta ^{2} \in \Lambda (D_{n})$, we can rewrite the
relation at the central vertex as
\begin{equation*}
(1+ \alpha )^{-1} (1+\beta )^{-1} (1+\gamma )^{-1} = 1 \quad
\implies \quad (1- \alpha ) (1-\beta ) = (1+\gamma ) \quad \implies
\quad \gamma = -\alpha -\beta + \alpha \beta .
\end{equation*}
Imposing this relation, we obtain a smaller spanning set consisting of
idempotents at the vertices and words in $\alpha $ and $\beta $. These
words necessarily alternate and hence modulo cyclic permutation, one has
a single length $4i$ generator $(\alpha \beta )^{i}$ for $i >0$.

Notice the path $\gamma ^{i}$ is contained in a single arm of
$D_{n}$ and hence is zero by Lemma~\ref{lem:_vanishing_cycles}. We conclude
that $0 = \gamma ^{i} = (-\alpha -\beta + \alpha \beta )^{i}$ for all
$i$. When $i=1$, $\alpha $ and $\beta $ also vanish by Lemma~\ref{lem:_vanishing_cycles}, so $\alpha \beta =0$. The result follows by
induction on $i$ since we can expand the leading term of the relation,
$(\alpha \beta )^{i}$, in terms of shorter-length terms, which are zero
by the inductive hypothesis.
\end{proof}

\begin{cor}
\label{cor:_HH0_Dn}
$\Lambda _{R}(D_{n})_{\text{cyc}} = R(D_{n})_{0}$ for any commutative ring
$R$.
\end{cor}

\begin{lem}
\label{lem:_length_4}
The length 4 cycle $\alpha \beta \in \Pi (Q)_{\text{cyc}}$ for
$Q= D_{n}, E_{6}, E_{7}, E_{8}$ is non-zero with
$2 \alpha \beta = 0$ and
$\alpha \beta = 0 \in \Lambda (Q)_{\text{cyc}}$.
\end{lem}

\begin{proof}
The statement in $\Pi (Q)_{\text{cyc}}$ holds by Remark~\ref{rem:_linearly_indep}. In $\Lambda (Q)$, we can rewrite the central
relation
%
\begin{align}
\label{eq:_leading_term_ba}
(1+\alpha )^{-1}(1+\beta )^{-1}(1+\gamma )^{-1} = 1 \quad &\implies
\quad 1 - \gamma + \gamma ^{2} - \cdots \pm \gamma ^{ \max \{5, n-3
\}} = 1 + \alpha + \beta + \beta \alpha
\nonumber
\\
& \implies \quad \gamma + \gamma ^{2} - \cdots \pm \gamma ^{ \max \{5,
n-3 \}} = \alpha + \beta + \beta \alpha .
\end{align}
Every term except $\beta \alpha $ is contained in one arm of $Q$ and hence
is zero in $\Lambda (Q)_{\text{cyc}}$ by Lemma~\ref{lem:_vanishing_cycles}. So $\alpha \beta = \beta \alpha =0$, as desired.
\end{proof}

For a ring $R$, write $R^{\times}$ for the set of invertible elements in $R$.
\begin{cor}
\label{cor:_no_iso_Dn}
Let $R$ be a commutative ring. If $2 \notin R^{\times}$ then
$\Lambda _{R}(D_{n}) \ncong \Pi _{R}(D_{n})$.
\end{cor}

\begin{proof}
By Corollary~\ref{cor:_HH0_Dn},
$\Lambda _{R}(D_{n})_{\text{cyc}} = R(D_{n})_{0}$ and if
$2 \notin R^{\times}$, then by Lemma~\ref{lem:_length_4},
$\alpha \beta \neq 0 \in \Pi _{R}(D_{n})_{\text{cyc}}$. We conclude that
$\Lambda _{R}(D_{n})_{\text{cyc}} \ncong \Pi _{R}(D_{n})_{
\text{cyc}}$ and hence
$\Lambda _{R}(D_{n}) \ncong \Pi _{R}(D_{n})$.
\end{proof}

Note that after taking the quotient by a maximal ideal containing
$2$ in $R$ the argument is sharper. Any corrected isomorphism sends
$\alpha \beta $ to itself plus higher order terms and
$\alpha \beta = 0 \in \Lambda _{k}(D_{n})_{\text{cyc}}$ while
$\alpha \beta \neq 0 \in \Pi _{k}(D_{n})_{\text{cyc}}$. So one need not
compute the remaining zeroth Hochschild homology.

\begin{rem}
\label{rem:_Pi(Dn)_cyc}
The first paragraph of the proof of Proposition~\ref{prop:_HH_0_D_n} implies
$\Pi (D_{n})_{\text{cyc}}$ is a quotient of $\mathbb{Z}$-span of
$\{ (\alpha \beta )^{i} \}_{i=1, \dots , \left \lfloor{ n/2-1}\right
\rfloor }$. The identity
$ 0 = \gamma ^{i} = (\alpha \beta )^{i} + (\beta \alpha )^{i} = 2(
\alpha \beta )^{i} \in \Pi (D_{n})_{\text{cyc}} $ implies these classes
are all 2-torsion. Each class is non-zero by
\cite[Theorem 13.1.1 (2)]{Schedler16}, see Remark~\ref{rem:_linearly_indep}, so
\begin{equation*}
\Pi (D_{n})_{\text{cyc}} \cong \mathbb{Z}(D_{n})_{0} \oplus \mathbb{Z}/
2 \mathbb{Z}^{>0}[ x]/ ( x^{ \lfloor n/2 \rfloor}),
\end{equation*}
as graded $\mathbb{Z}$-modules with $|x| = 4$.
\end{rem}

\begin{lem}
\label{lem:_length_6}
The length 6 cycle $\beta \alpha \beta \in \Pi (Q)_{\text{cyc}}$ for
$Q$ type E is non-zero with $3 \beta \alpha \beta = 0$ and
$\beta \alpha \beta = 0 \in \Lambda (Q)_{\text{cyc}}$.
\end{lem}

\begin{proof}
The statement in $\Pi (Q)_{\text{cyc}}$ holds by Remark~\ref{rem:_linearly_indep}. In $\Lambda (Q)$, the relation
%
\begin{align}
\label{eq:_leading_term_bab}
(1+\alpha )^{-1}(1+\beta )^{-1}(1+\gamma )^{-1} = 1 \quad &\implies
\quad (1-\alpha )(1-\beta +\beta ^{2}) = (1+ \gamma )
\nonumber
\\
& \implies \quad \gamma = -\alpha - \beta + \alpha \beta +\beta ^{2} -
\alpha \beta ^{2},
\end{align}
has terms $\gamma , \alpha , \beta $, and $\beta ^{2}$ entirely contained
in one arm and hence zero in $\Lambda (Q)_{\text{cyc}}$. The
$\alpha \beta $ term is zero by Lemma~\ref{lem:_length_4}. We conclude
that
$\beta \alpha \beta = \alpha \beta ^{2} = 0 \in \Lambda (Q)_{
\text{cyc}}$.
\end{proof}

\begin{lem}
\label{lem:_length_8}
Let $2 \leq d \leq 8$. Any length $d$ cycle in $\alpha $ and
$\beta $ is zero in $\Lambda (E_{8})_{\text{cyc}}$.
\end{lem}

\begin{proof}
By Lemma~\ref{lem:_spanning_set_general},
$\alpha = 0 = \beta \in \Lambda (E_{8})_{\text{cyc}}$. In length 4 and 6
the result follows from Lemmas~\ref{lem:_length_4} and \ref{lem:_length_6} respectively. In length 8, square both sides of Equation
(\ref{eq:_leading_term_ba}). This gives an equation with left-hand side zero
in $\Lambda (E_{8})_{\text{cyc}}$ by Lemma~\ref{lem:_vanishing_cycles} and
right-hand side $\beta \alpha \beta \alpha $ plus length 4 and 6 terms,
which are zero. We conclude that
$\beta \alpha \beta \alpha = 0 \in \Lambda (E_{8})_{\text{cyc}}$. The result
follows as the remaining length 8 paths are
\begin{equation*}
\alpha \beta \alpha \beta = \beta \alpha \beta \alpha = 0 , \ \
\alpha \beta ^{2} \alpha = \alpha ^{2} \beta ^{2} = 0, \ \ \beta ^{2}
\alpha \beta = \beta \alpha \beta ^{2} = \alpha \beta ^{3} = 0.\qedhere
\end{equation*}
\end{proof}

\begin{lem}
\label{lem:_length_10}
The length 10 cycle
$\beta \alpha \beta \alpha \beta \in \Pi (E_{8})_{\text{cyc}}$ is non-zero
with $5 \beta \alpha \beta \alpha \beta = 0$ while
$\beta \alpha \beta \alpha \beta \in \Lambda (E_{8})_{\text{cyc}}$ is zero.
\end{lem}

\begin{proof}
The statement in $\Pi (Q)_{\text{cyc}}$ holds by Remark~\ref{rem:_linearly_indep}. To show
$\beta \alpha \beta \alpha \beta = 0 \in \Lambda (E_{8})_{\text{cyc}}$,
notice
\begin{align*}
0 = \gamma ^{2} &= \alpha \beta ^{2} \alpha \beta ^{2} - \beta ^{2}
\alpha \beta ^{2} - \alpha \beta ^{2} \alpha \beta - \alpha \beta
\alpha \beta ^{2} + \text{(length at most 8 paths)}
\\
& = \alpha \beta ^{2} \alpha \beta ^{2} - 2 \beta \alpha \beta
\alpha \beta + \text{(length at most 8 paths)}
\\
\beta \alpha \beta \alpha \beta \alpha &= \alpha \beta ^{2} \alpha
\beta ^{2} - \beta ^{2} \alpha \beta ^{2} - \beta ^{2} \alpha \beta
\alpha - \beta \alpha \beta ^{2} \alpha - \beta \alpha \beta \alpha
\beta - \alpha \beta ^{2} \alpha \beta - \alpha \beta \alpha \beta ^{2}
- \alpha \beta \alpha \beta \alpha
\\
&= \alpha \beta ^{2} \alpha \beta ^{2} - 5 \beta \alpha \beta \alpha
\beta
\end{align*}
so by Lemma~\ref{lem:_length_8}, it suffices to show
$\alpha \beta ^{2} \alpha \beta ^{2}$ and
$\beta \alpha \beta \alpha \beta \alpha =0 \in \Lambda (E_{8})_{
\text{cyc}}$. Using the relation
\begin{equation*}
\beta ^{2} \alpha \beta ^{2} \alpha = \alpha \beta \alpha \beta
\alpha \beta + \beta ^{2} \alpha \beta ^{2} + \beta ^{2} \alpha
\beta \alpha + \beta \alpha \beta ^{2} \alpha + \beta \alpha \beta
\alpha \beta + \alpha \beta ^{2} \alpha \beta + \alpha \beta \alpha
\beta ^{2} + \alpha \beta \alpha \beta \alpha
\end{equation*}
in $\Lambda (E_{8})$, notice the three commutators:
\begin{itemize}
\item[(1)]
$[\beta ^{2} \alpha \beta ^{2} \alpha , \beta ^{2}] =\alpha \beta
\alpha \beta \alpha \beta ^{2} + \beta ^{2} \alpha \beta \alpha
\beta ^{2} + \beta \alpha \beta ^{2} \alpha \beta ^{2} = \alpha
\beta \alpha \beta \alpha \beta ^{2}$
\item[(2)]
$[ \alpha \beta ^{2} \alpha \beta ^{2}, \alpha ] = \alpha \beta ^{2}
\alpha \beta ^{2} + \alpha \beta ^{2} \alpha \beta \alpha + \alpha
\beta \alpha \beta ^{2} \alpha + \alpha \beta \alpha \beta \alpha
\beta = \alpha \beta ^{2} \alpha \beta ^{2} +\alpha \beta \alpha
\beta \alpha \beta $
\item[(3)]
$[ \beta ^{2} \alpha \beta ^{2} \alpha , \beta ] = \alpha \beta
\alpha \beta \alpha \beta ^{2} + \beta ^{2} \alpha \beta \alpha
\beta + \beta \alpha \beta ^{2} \alpha \beta + \beta \alpha \beta
\alpha \beta ^{2} + \alpha \beta ^{2} \alpha \beta ^{2} + \alpha
\beta \alpha \beta \alpha \beta$ \\
$\phantom{[ \beta ^{2} \alpha \beta ^{2} \alpha , \beta ]}
= \alpha \beta \alpha \beta \alpha \beta ^{2} + \beta \alpha \beta ^{2}
\alpha \beta + \alpha \beta ^{2} \alpha \beta ^{2} + \alpha \beta
\alpha \beta \alpha \beta $
\end{itemize}
implies
\begin{equation*}
\text{(3)}-\text{(2)}-\text{(1)}: [ \beta ^{2} \alpha \beta ^{2} \alpha ,
\beta ] - [ \alpha \beta ^{2} \alpha \beta ^{2}, \alpha ] - [\beta ^{2}
\alpha \beta ^{2} \alpha , \beta ^{2}] = \alpha \beta ^{2} \alpha
\beta ^{2}.
\end{equation*}
Hence by (2), both $\alpha \beta ^{2} \alpha \beta ^{2}$ and
$\beta \alpha \beta \alpha \beta \alpha = 0 \in \Lambda (E_{8})_{
\text{cyc}}$, which implies
$\beta \alpha \beta \alpha \beta = 0 \in \Lambda (E_{8})_{\text{cyc}}$.
\end{proof}

\begin{prop}
\label{prop:_HH_0_Em}
$\Lambda (E_{m})_{\text{cyc}} \cong \mathbb{Z}(E_{m})_{0}$ for
$m=6, 7, 8$.
\end{prop}

\begin{proof}
The surjections
$\Lambda (E_{8})_{\text{cyc}} \twoheadrightarrow \Lambda (E_{7})_{
\text{cyc}} \twoheadrightarrow \Lambda (E_{6})_{\text{cyc}}$, imply the result
for $E_{6}$ and $E_{7}$ follows from $E_{8}$. We can identify
\begin{equation*}
e_{3} \Lambda (E_{8}) e_{3} \cong
\frac{\mathbb{Z}\langle \alpha , \beta , \gamma \rangle}{(\alpha ^{2}, \beta ^{3}, \gamma ^{5}, \gamma = -\alpha -\beta + \alpha \beta -\beta ^{2} + \alpha \beta ^{2})}
\cong
\frac{\mathbb{Z}\langle x, y \rangle}{(x^{2}, y^{3}, (x + y - xy-y^{2}+xy^{2})^{5})}.
\end{equation*}
Using the lexicographical ordering with $y > x$, this algebra has a Gr\"{o}bner
basis with leading terms
\begin{equation*}
x^{2}, \ \ y^{3}, \ \  yxyxyx, \ \  y^{2}xy^{2}x, \ \  y^{2}xyxy^{2}xy^{2},
\ \  yxyxy^{2}xyxy, \ \  yxy^{2}xyxy^{2}xyx
\end{equation*}
and a basis:
\begin{align*}
\{ 1&, \ \  x, \ \  y, \ \  yx, \ \  xy, \ \  y^{2}, \ \  xyx, \ \  y^{2}x,
\ \  yxy, \ \  xy^{2}, \ \  yxyx, \ \  xy^{2}x, \ \  xyxy, \ \
y^{2}xy, \ \  yxy^{2}, \ \  xyxyx,
\\
& y^{2}xyx, \ \  yxy^{2}x, \ \  yxyxy, \ \  xy^{2}xy, \ \  xyxy^{2},
\ \  y^{2}xy^{2}, \ \
xy^{2}xyx, \ \  xyxy^{2}x, \ \  xyxyxy, \ \  y^{2}xyxy,
\\
& yxy^{2}xy, \ \
 yxyxy^{2}, \ \  xy^{2}xy^{2}, \ \
yxy^{2}xyx, \ \  yxyxy^{2}x, \ \  xy^{2}xyxy, \ \  xyxy^{2}xy, \ \  xyxyxy^{2},
\ \  y^{2}xyxy^{2},
\\
& yxy^{2}xy^{2}, \ \ xyxy^{2}xyx, \ \  xyxyxy^{2}x, \ \  y^{2}xyxy^{2}x,
\ \  yxy^{2}xyxy, \ \  yxyxy^{2}xy, \ \  xy^{2}xyxy^{2}, \ \  xyxy^{2}xy^{2},
\\
&yxyxy^{2}xyx, \ \ xy^{2}xyxy^{2}x, \ \  xyxy^{2}xyxy, \ \  xyxyxy^{2}xy,
\ \  y^{2}xyxy^{2}xy, \ \  yxy^{2}xyxy^{2}, \ \  yxyxy^{2}xy^{2},
\\
&xyxyxy^{2}xyx, \ \ y^{2}xyxy^{2}xyx, \ \  yxy^{2}xyxy^{2}x, \ \  xy^{2}xyxy^{2}xy,
\ \  xyxy^{2}xyxy^{2}, \ \  xyxyxy^{2}xy^{2},
\\
&xy^{2}xyxy^{2}xyx, \ \  xyxy^{2}xyxy^{2}x, \ \  yxy^{2}xyxy^{2}xy,
\ \
xyxy^{2}xyxy^{2}xy \}.
\end{align*}
Modulo cyclic permutation, there are eight elements not containing the
leading term of an element in the Gr\"{o}bner basis:
\begin{align*}
\{ 1, \  x, \  y, \  yx = xy, \  y^{2}, \  y^{2}x = yxy = xy^{2}, \  yxyx
= xyxy, \  y^{2}xyx = yxy^{2}x = yxyxy = xy^{2}xy = xyxy^{2} \}.
\end{align*}
By Lemma~\ref{lem:_length_8} and Lemma~\ref{lem:_length_10} all but
$1$ are necessarily zero.
\end{proof}

\begin{cor}
\label{cor:HH_0_Em_R}
$\Lambda_R(E_m)_{\text{cyc}} = R(E_m)_0$ for any commutative ring $R$, and $m=6, 7, 8$.
\end{cor}

\begin{cor}
\label{cor:_no_iso_E67}
Let $R$ be a commutative ring and let $m=6, 7, 8$. Then,
$\Lambda _{R}(E_{m}) \ncong \Pi _{R}(E_{m})$ if
$2, 3 \notin R^{\times}$. And
$\Lambda _{R}(E_{8}) \ncong \Pi _{R}(E_{8})$ if
$5 \notin R^{\times}$.
\end{cor}

\begin{proof}
By Corollary~\ref{cor:HH_0_Em_R}
$\Lambda _{R}(E_{m})_{\text{cyc}} = R(E_{m})_{0}$ while by Lemma~\ref{lem:_length_4} (respectively Lemma~\ref{lem:_length_6} and Lemma~\ref{lem:_length_10}) $\Pi _{R}(E_{m})_{\text{cyc}}$ is strictly larger
if $2 \notin R^{\times}$ (respectively $3 \notin R^{\times}$ and
$5 \notin R^{\times}$).
\end{proof}

As before, reducing modulo a maximal ideal containing $3$ (respectively
$5$) and using Corollary~\ref{cor:_correct_iso} the argument only relies
on computing the single class in Lemma~\ref{lem:_length_6} (respectively
Lemma~\ref{lem:_length_10}).

\section{Construction of isomorphisms}
\label{sec:_construct_isos}

Shaw constructs an explicit isomorphism
\begin{equation*}
\Lambda _{\mathbb{Z}[1/2]}(D_{4}) \cong \Pi _{\mathbb{Z}[1/2]}(D_{4})
\qquad\quad
\begin{array}{c}
a \mapsto a \left (1-\frac{1}{2} b^{*} b \right ) , \ a^{*} \mapsto
\left (1+\frac{1}{2} b^{*} b \right )a^{*}
\\
a' \mapsto a' \text{ for } a' \in \{ e_{i}, b, b^{*}, c, c^{*} \}_{i = 1,
2, 3, 4}.
\end{array}
\end{equation*}
He then laments ``However, this example doesn't really suggest how this
question can be answered in general, because it is not practical to attempt
this analysis for the larger Dynkin quivers.'' \cite[pg. 128]{Shaw05} Practicality
aside, this is precisely what we do.

For type A,
$\Lambda _{\mathbb{Z}}(A_{n}) = \Pi _{\mathbb{Z}}(A_{n})$ follows from
Proposition~\ref{prop:_partial_equal}. In type D, Shaw's map extends to
a uniform isomorphism over $\mathbb{Z}[1/2]$. In type E, finding an isomorphism
is more cumbersome, as we now explain.%

Let $R$ be any commutative ring. Note that the space of corrected functions
$\overline{Q}_1 \rightarrow \Pi _{R}(Q)$ (see Definition~\ref{def:_corrected}) can be identified with
\begin{equation*}
\text{Hom}_{\text{corr}}( \overline{Q}_1, \Pi _{R}(Q) ) \cong \bigoplus _{a
\in \overline{Q}_{1}} e_{s(a)} \Pi ^{\geq 3}_{R}(Q) e_{t(a)}
\end{equation*}
where each arrow $a \in \overline{Q}_{1}$ is sent to itself plus an element
in $e_{s(a)} \Pi ^{\geq 3}_{R}(Q) e_{t(a)}$. A choice of $R$-basis for
$\oplus _{a \in \overline{Q}_{1}} e_{s(a)} \Pi ^{\geq 3}_{R}(Q) e_{t(a)}$
is a choice of identification with the affine space
$\mathbb{A}^{N}_{R}$ of dimension $N$. Note that $N=50, 120$, and
$354$ for $Q= E_{6}, E_{7}$, and $E_{8}$ respectively. These corrected
functions extend uniquely to a corrected $R$-algebra map
$R \overline{Q} \rightarrow \Pi _{R}(Q)$. Hence the space of corrected
isomorphisms $\varphi : \Lambda _{R}(Q) \rightarrow \Pi _{R}(Q)$ can be
identified with the vanishing set of a single element in
$\Pi _{R}(Q)$, namely the image of $r_{\text{mult}}$. Expanding
$\varphi (r_{\text{mult}})$ in a basis for
$\oplus _{i \in Q_{0}} e_{i} \Pi ^{\geq 4}_{R}(Q) e_{i}$, gives a system
of $M$ equations in $\mathbb{A}^{N}_{R}$, where $M = 22$, $61$, and
$178$ for $Q= E_{6}$, $E_{7}$, and $E_{8}$ respectively.

Since we can uniformly write the central relation
$e_{3} r_{\text{mult}}$ as
\begin{equation*}
(1+ \alpha )^{-1} (1+ \beta )^{-1} (1+ \gamma )^{-1} = 1 \implies (1-
\alpha ) = (1+ \gamma )(1+\beta ) \implies 0 = \alpha + \beta +
\gamma + \gamma \beta
\end{equation*}
these equations are at most quartic and, by design, have a linear term
that can be used for substitutions. After tediously carrying out these
substitutions and inverting bad primes, one can find a solution by exploiting
the fact that many variables can be taken to be zero. Despite the difficulty
in finding these maps, one can check they are isomorphisms more easily.
One simply needs to show that the image of $r_{\text{mult}}$ is in the ideal
$I := (r_{\text{add}})$. In the appendix we show this using Magma.

\subsection{Type D}
\label{sec5.1}

Define $p(x): = \sum _{i=0}^{n-3} \left ( \frac{-x}{2} \right )^{i}$. Denote
by
$p^{\text{even}}(x) := \sum _{i=0, i \ \text{even}}^{n-3} \left (
\frac{-x}{2} \right )^{i}$ and similarly by $p^{\text{odd}}(x)$ the odd
part of $p(x)$. We give the following basic facts:

\begin{prop}
\label{prop:_poly_facts}
Consider $p(x)$ as an element of $\mathbb{Z}[1/2][x]/(x^{n-2})$.
\begin{itemize}
\item[(1)] $p(x)$ has inverse $q(x) := 1+ \frac{1}{2} x$. Hence
$p(-x)$ has inverse $q(-x)$.
\item[(2)] $p^{\text{odd}}(x) = -\frac{1}{2} x p^{\text{even}}(x)$ which
implies $x p^{\text{even}}(x) = p(-x)- p(x)$. Hence
$x p^{\text{even}}(x)+p(x) = p(-x)$.
\end{itemize}
\end{prop}

\begin{proof}
(1) The second statement follows from the first replacing $x$ with
$-x$. And indeed the product
\begin{equation*}
p(x) q(x) = \sum _{i=0}^{n-3} \left ( \frac{-x}{2} \right )^{i}
\left (1+\frac{1}{2} x \right ) = \sum _{i=0}^{n-3} \left (
\frac{-x}{2} \right )^{i} - \left ( \frac{-x}{2} \right )^{i+1} = 1 -
\left (\frac{-x}{2}\right )^{n-2} = 1.
\end{equation*}
(2) The first statement is clear from the definition of $p$. The remaining parts
follow from the fact that the odd part of a function of $f$ can be written
as $f^{\text{odd}}(x) = \frac{f(x) - f(-x)}{2}$, hence
$-2f^{\text{odd}}(x) = f(-x) -f(x)$.
\end{proof}

Recall, we label the $D_{n}$ quiver
\begin{equation*}
\xymatrix{ 1 \ar[rd]^a & & & & \\
D_{n} = \qquad\quad & 3 & 4 \ar[l]_{c_{1}} & \ar [l]_{c_{2}} \cdots & \ar [l]_{c_{n-3}} n \\
2 \ar[ru]_b & & & &}
\end{equation*}
and choose an ordering on the arrows $a^{*} < b^{*} < c_{1}^{*}$ with source
3. We write $\gamma _{i} := c_{i}^{*} c_{i}$ and write
$\gamma := \gamma _{1}$.

\begin{prop}
\label{prop:_iso_Dn}
Let $Q = D_{n}$ with vertices and arrows labelled as above. Define
$\phi : \mathbb{Z}[1/2] \overline{Q} \rightarrow \Pi _{\mathbb{Z}[1/2]}(Q)$
on generators by
\begin{equation*}
\phi (a) = a p(\gamma ) \quad \phi (a^{*}) = q(\gamma ) a^{*} \quad
\phi (b) = b \quad \phi (b^{*}) = b^{*} \quad \phi (c_{i}) = c_{i} p(-
\gamma _{i}) \quad \phi (c_{i}^{*}) = c_{i}^{*}
\end{equation*}
where $p(x): = \sum _{i=0}^{n-3} \left ( \frac{-x}{2} \right )^{i}$ and
$q(x) = 1+ \frac{1}{2} x$. Then $\phi $ descends to an isomorphism on the
quotient
$\Lambda _{\mathbb{Z}[1/2]}(Q) \rightarrow \Pi _{\mathbb{Z}[1/2]}(Q)$.
\end{prop}

\begin{proof}
The additive preprojective relations are:
\begin{equation*}
\alpha + \beta + \gamma , \quad aa^{*}, \quad bb^{*}, \quad c_{n-3} c_{n-3}^{*},
\quad c_{i}c_{i}^{*} - c_{i+1}^{*} c_{i+1} \ i = 1, \dots , n-4
\end{equation*}
so in particular $\alpha ^{2} = \beta ^{2} = \gamma ^{n-2}=0$ and we can
use the results of Proposition~\ref{prop:_poly_facts}. The multiplicative
preprojective relations are:
\begin{equation*}
-\alpha \beta + \alpha + \beta + \gamma , \quad aa^{*}, \quad bb^{*},
\quad c_{n-3} c_{n-3}^{*}, \quad c_{i}c_{i}^{*} - c_{i+1}^{*} c_{i+1}
\ i = 1, \dots , n-4.
\end{equation*}
Notice $\phi $ descends to a map on the multiplicative preprojective algebra
since:
\begin{align*}
\phi (a a^{*}) &= a p(\gamma ) q(\gamma ) a^{*} = a a^{*} - a \left (
\frac{-\gamma}{2} \right )^{n-2} a^{*} \equiv 0
\\
\phi (b b^{*}) &= b b^{*} \equiv 0
\\
\phi ( c_{i} c_{i}^{*}) &= c_{i} p(-\gamma _{i}) c_{i}^{*} \equiv
\begin{cases}
c_{i+1}^{*} c_{i+1} p(-\gamma _{i}) = \phi ( c_{i+1}^{*} c_{i+1}) &
\text{if } i<n-3
\\
0 & \text{if } i = n-3.
\end{cases}
\end{align*}

It remains to show $\phi(-\alpha \beta + \alpha + \beta + \gamma) \equiv \alpha + \beta + \gamma \equiv 0$. 
First notice, we have the identities 
\begin{align}
0 = \alpha + \beta +\gamma \quad & \implies \quad \beta = -\alpha - \gamma \label{eq: elim beta} \\
0 = (\beta)^{2} = (-\alpha - \gamma)^{2} = \alpha \gamma + \gamma \alpha + \gamma^2 \quad &\implies \quad \gamma \alpha = -\alpha \gamma - \gamma^{2}   \label{eq: identities}
\end{align}
and hence we will rewrite words in $\alpha$, $\beta$, and $\gamma$ as linear combinations of the paths $\{ \alpha \gamma^{j} \}_{j=0, \dots, n-3}$.
Next observe that 
\begin{align}
\gamma^{2} \alpha &= \gamma[ -\alpha \gamma - \gamma^2] = \alpha \gamma^{2} + \gamma^{3} - \gamma^{3} = \alpha \gamma^{2}   \label{eq: commuting alpha} \\
\gamma^{2} \beta &= \beta \gamma^{2}  \label{eq: commuting beta}
\end{align}
and 
\begin{equation} \label{eq: odd vanishing}
\alpha \gamma^{2i+1} \beta = \alpha [-\alpha  -\beta]^{2i+1}\beta = \alpha [\alpha (\beta \alpha)^i + (\beta \alpha)^{i} \beta ] \beta = 0.
\end{equation}
 Therefore, 
\begin{align}
\phi(\alpha \beta) = \phi(\alpha) \phi(\beta) &= q(\gamma) \alpha p(\gamma) \beta \overset{\ref{eq: odd vanishing}}{=} q(\gamma) \alpha p^{\even}(\gamma) \beta \nonumber \\ 
&\overset{\ref{eq: commuting beta}}{=} q(\gamma) \alpha \beta  p^{\even}(\gamma) \overset{\ref{eq: elim beta}}{=} - q(\gamma) \alpha \gamma p^{\even}(\gamma).  \label{eq: alpha beta term}
\end{align}
Hence, since $\phi(\gamma) = \gamma p(-\gamma)$, we have:
\begin{align*}
\phi(-\alpha \beta + \alpha +  \beta + \gamma) &\overset{\ref{eq: alpha beta term}}{=} q(\gamma) \alpha \gamma p^{\even}(\gamma) + q(\gamma) \alpha p(\gamma) +\beta + \gamma p(-\gamma) \\
&=q(\gamma) \alpha [\gamma p^{\even}(\gamma) +  p(\gamma)] +\beta + \gamma p(-\gamma) \\
\footnotesize{\text{Prop } \ref{prop: poly facts} \text{ (2)}} &= q(\gamma) \alpha [ p(-\gamma)] + \beta + \gamma p(-\gamma) \\
&= \left( 1 + \frac{1}{2} \gamma \right )  \alpha [ p(-\gamma)] + \beta + \gamma p(-\gamma) \\
 &\overset{\ref{eq: identities}}{=}  \alpha q(-\gamma) [ p(-\gamma) ] - \frac{1}{2} \gamma^2 [p(-\gamma)] + \beta + \gamma p(-\gamma) \\
\footnotesize{\text{Prop } \ref{prop: poly facts} \text{(1)}} &=  \alpha+\beta  + \gamma \left [ -\frac{1}{2} \gamma p(-\gamma) + p(-\gamma) \right ] \\
&=  \alpha+\beta  + \gamma \left [ q(-\gamma) p(-\gamma)  \right ] \\
\footnotesize{\text{Prop } \ref{prop: poly facts} \text{(1)}} &= \alpha + \beta +\gamma.
\end{align*}

So $\phi $ descends to a map
$\psi : \Lambda _{\mathbb{Z}[1/2]}(Q) \rightarrow \Pi _{\mathbb{Z}[1/2]}(Q)$.
Notice $\psi $ takes each path to itself plus a $\mathbb{Z}[1/2]$-linear
combination of longer paths. This implies $\psi $ is invertible as we now
explain.

Write $\Pi (Q) := \Pi _{\mathbb{Z}[1/2]}(Q)$,
$\Lambda (Q) := \Lambda _{\mathbb{Z}[1/2]}(Q)$ and denote by
$\Pi ^{\geq d}(Q)$ the vector subspace of $\Pi (Q)$ spanned by paths of
length at least $d$. Observe that the composition
\begin{equation*}
\Lambda (Q) \overset{\psi}{\rightarrow} \Pi (Q)
\overset{\text{gr}}{\rightarrow} \oplus _{i=0}^{2n} \Pi ^{\geq i}(Q)/
\Pi ^{\geq i+1}(Q)
\end{equation*}
is surjective and hence $\psi $ is surjective. For injectivity, notice
that the projection
$k \overline{Q} \rightarrow \text{gr}( \Lambda (Q))$ factors through
$\pi : \Pi (Q) \rightarrow \text{gr}( \Lambda (Q))$ since
$\text{gr}(r_{\text{mult}}) = r_{\text{add}}$. The composition
\begin{equation*}
\text{gr}(\psi ) \circ \pi : \Pi (Q) \rightarrow \text{gr}(\Lambda (Q))
\rightarrow \text{gr}(\Pi (Q)) \cong \Pi (Q)
\end{equation*}
is the identity. Since $\pi $ is surjective, $\text{gr}(\psi )$ is injective
and hence both are isomorphisms. We conclude that $\psi $ is injective
and hence a $\mathbb{Z}[1/2]$-algebra isomorphism.
\end{proof}

\begin{cor}
\label{cor5.3}
Let $R$ be a commutative ring and $n \geq 4$. Then
$\Lambda _{R}(D_{n}) \cong \Pi _{R}(D_{n})$ if and only if
$2 \in R^{\times}$.
\end{cor}

\begin{proof}
If $2 \in R^{\times}$ then the unique (unital) ring map
$\mathbb{Z}\rightarrow R$ factors through $\mathbb{Z}[1/2]$, giving a non-zero
map $\mathbb{Z}[1/2] \rightarrow R$. Hence one can regard the coefficients
in the map $\psi $ in Proposition~\ref{prop:_iso_Dn} as living in
$R$. So $\psi $ defines an isomorphism of $R$-algebras
$\Lambda _{R}(D_{n}) \cong \Pi _{R}(D_{n})$. Corollary~\ref{cor:_no_iso_Dn} establishes the converse.
\end{proof}

\subsection{Type E}
\label{sec5.2}

For ease of notation, in this section $S := \mathbb{Z}[1/2, 1/3]$.
%
\begin{prop}
\label{prop:_iso_E6}
The map $\varphi : S\overline{E}_{6} \rightarrow \Pi _{S}(E_{6})$ defined
to be the identity on all vertices and arrows except:
\begin{equation*}
\varphi (b) = b r(\alpha , \gamma ) , \quad \varphi (b^{*}) = s(
\alpha , \gamma ) b^{*} \quad \varphi (c) = ct(\alpha , \gamma ),
\quad \varphi (c^{*}) = u(\alpha , \gamma )c^{*}
\end{equation*}
where
\begin{align*}
r(\alpha , \gamma ) &:= (1-1/2\alpha -1/12 \alpha \gamma -1/24
\alpha \gamma ^{2}) , \quad s(\alpha , \gamma ) := (1-1/2 \gamma -1/6
\alpha \gamma )
\\
t(\alpha , \gamma ) &:= (1- 1/12 \gamma \alpha ), \quad u(\alpha ,
\gamma ) :=(1+1/2 \gamma + 1/12 \gamma \alpha - 1/24 \alpha \gamma
\alpha ),
\end{align*}
descends to an isomorphism
$\Lambda _{S}(E_{6}) \cong \Pi _{S}(E_{6})$.
\end{prop}

\begin{proof}
We need to show $\varphi (r_{\text{mult}}) = 0$ and hence in this case it
suffices to show
\begin{equation*}
b r(\alpha , \gamma ) s(\alpha , \gamma ) b^{*} - d^{*}d = 0, \quad c t(
\alpha , \gamma ) u(\alpha , \gamma ) c^{*} - e^{*} e = 0
\end{equation*}
\begin{equation*}
\alpha + s(\alpha , \gamma ) \beta r(\alpha , \gamma ) + u(\alpha ,
\gamma ) \gamma t(\alpha , \gamma ) + u(\alpha , \gamma ) \gamma t(
\alpha , \gamma )s(\alpha , \gamma ) \beta r(\alpha , \gamma ) = 0.
\end{equation*}
For the first relation notice
\begin{equation*}
r(\alpha , \gamma )s(\alpha , \gamma ) = 1/144 \alpha \beta \alpha
\beta ^{2} + 1/72 \alpha \beta \alpha \beta + 1/2 \beta + 1
\end{equation*}
and hence
$b r(\alpha , \gamma ) s(\alpha , \gamma ) b^{*} = b b^{*}$ since
$b \beta b^{*} = 0$ and $b \alpha \beta \alpha \beta b^{*} =0$. Similarly,
\begin{equation*}
t(\alpha , \gamma ) u(\alpha , \gamma ) = u(\alpha , \gamma ) t(
\alpha , \gamma ) = 1/144 \gamma \alpha \gamma ^{2} + 1/144 \gamma ^{2}
\alpha \gamma + 1/24 \alpha \gamma ^{2} + 1/24 \gamma \alpha \gamma + 1/2
\gamma + 1
\end{equation*}
and hence $c t(\alpha , \gamma ) u(\alpha , \gamma ) c^{*} = cc^{*}$. Finally,
one can compute
\begin{equation*}
u(\alpha , \gamma ) \gamma t(\alpha , \gamma ) + s(\alpha , \gamma )
\beta r(\alpha , \gamma ) = \gamma + \beta + u(\alpha , \gamma )
\gamma t(\alpha , \gamma ) s(\alpha , \gamma ) \beta r(\alpha ,
\gamma )
\end{equation*}
as desired, with each side equal to\nopagebreak
\begin{equation*}
1/24 \gamma ^{2}\alpha \gamma ^{2} + 1/12 \gamma ^{2} \alpha \gamma + 1/2
\gamma ^{2} \alpha + \gamma \beta + \gamma + \beta .
\end{equation*}
We conclude that $\varphi $ descends to a map
$\Lambda _{S}(E_{6}) \cong \Pi _{S}(E_{6})$, which is an isomorphism as
argued in the proof of Proposition~\ref{prop:_iso_Dn} since
$r(\alpha , \gamma ), s(\alpha , \gamma ), t(\alpha , \gamma )$, and
$u(\alpha , \gamma )$ are all invertible.
\end{proof}

\begin{prop}
\label{prop:_iso_E7}
$\Lambda _{S}(E_{7}) \cong \Pi _{S}(E_{7})$.
\end{prop}

\begin{proof}
We claim that the map
$\phi : S \overline{E_{7}} \rightarrow \Pi _{S}(E_{7})$ given by the identity
on vertices and the arrows $a, a^{*}, d, d^{*}, f, f^{*}$ and sending
\begin{align*}
e &\mapsto 1/4ecc^{*} + e \quad e^{*} \mapsto 1/4cc^{*}e^{*} + e^{*}
\\
b &\mapsto -1/864b(\alpha \gamma )^{3}\alpha + 1/432b(\alpha \gamma )^{3}
+ 1/288b(\alpha \gamma )^{2}\alpha - 1/144b(\gamma \alpha )^{2}
\\
& \quad - 1/144b\alpha \gamma ^{2}\alpha - 1/144b(\alpha \gamma )^{2} +
1/24b\alpha \gamma \alpha - 1/12b\alpha \gamma - 1/2b\alpha + b;
\\
b^{*} &\mapsto 1/144\alpha \gamma \alpha \gamma b^{*} - 1/6\alpha
\gamma b^{*} - 1/2\gamma b^{*} + b^{*}
\\
c &\mapsto -1/6912c(\alpha \gamma )^{3}\alpha + 1/432c(\alpha \gamma )^{2}
\gamma \alpha + 7/3456c(\alpha \gamma )^{3} - 1/288c\gamma \alpha
\gamma ^{2}\alpha
\\
& \quad - 1/288c(\alpha \gamma )^{2}\gamma + 1/144c\gamma \alpha
\gamma ^{2} - 1/72c\alpha \gamma ^{2}\alpha - 1/144c(\alpha \gamma )^{2}
\\
& \quad + 1/24c\gamma \alpha \gamma + 1/24c\alpha \gamma ^{2} + 1/24c
\alpha \gamma \alpha + 1/12c\gamma ^{2} - 1/12c\gamma \alpha + 1/2c
\gamma + c
\\
c^{*} &\mapsto -1/3456\gamma \alpha \gamma \alpha \gamma \alpha c^{*} +
1/12\gamma ^{2}c^{*} + 1/12\gamma \alpha c^{*} + c^{*}
\end{align*}
descends to an isomorphism
$\Lambda _{S}(E_{7}) \cong \Pi _{S}(E_{7})$. The fact that the map descends
is verified in the appendix and its form ensures that the map is an isomorphism.
\end{proof}

\begin{cor}
\label{cor5.6}
Let $R$ be a commutative ring and let $m=6, 7$. Then
$\Lambda _{R}(E_{m}) \cong \Pi _{R}(E_{m})$ if and only if
$2, 3 \in R^{\times}$.
\end{cor}

\begin{proof}
The assumption that $2, 3 \in R^{\times}$ implies that $R$ receives a non-zero
map from $S$, and allows one to view the coefficients of $\varphi $ in
Proposition~\ref{prop:_iso_E6} and the coefficients of $\phi $ in Proposition~\ref{prop:_iso_E7} in $R$, giving the isomorphisms. Corollary~\ref{cor:_no_iso_E67} establishes the converse.
\end{proof}

\begin{prop}
\label{prop:_iso_E8}
Let $S' :=\mathbb{Z}[1/2, 1/3, 1/5]$. Then
$\Lambda _{S'}(E_{8}) \cong \Pi _{S'}(E_{8})$.
\end{prop}

\begin{proof}
To condense notation write $C := \alpha \gamma $. Define
$\varphi : \overline{E_{8}} \rightarrow \Pi _{S'}(E_{8})$ to be the identity
on vertices and the arrows $g, g^{*}$ and sending
\begin{align*}
a \mapsto p &:= 53383/93312000a\gamma C^{6} + 461/1036800a\gamma C^{2}
\gamma C\alpha + 199/129600a\gamma C^{3}\alpha + 1/4320a\gamma C^{2}
\gamma \alpha + a
\\
a^{*} \mapsto q &:= -53383/93312000\gamma C^{6}a^{*} - 26639/24883200C^{6}a^{*}
+ 259/691200C^{5}\gamma a^{*} + 1/2073600C^{5}a^{*}
\\
&+ 91/518400C^{4}\gamma a^{*} -1/3456C^{2}\gamma Ca^{*} + 7/17280C^{3}
\gamma a^{*} + 11/2880C^{3}a^{*} - 1/720C\gamma Ca^{*}
\\
&+ 1/720C^{2}\gamma a^{*} +1/144C^{2}a^{*} + 1/12Ca^{*} + a^{*}
\end{align*}
\begin{align*}
b \mapsto r &:= 773/9331200bC^{6}\alpha - 307/2488320bC^{5}\alpha - 17/172800bC^{3}
\gamma C\alpha + 2261/3110400bC^{4}\gamma \alpha
\\
&+ 749/3110400bC^{5} + 301/1036800b\gamma C^{2}\gamma C\alpha + 371/518400bC^{4}
\alpha - 1/17280b\gamma C^{2}\gamma C
\\
&+ 17/103680bC^{2}\gamma C\alpha - 37/57600bC^{3}\gamma \alpha - 13/28800bC^{4}
+ 11/17280b\gamma C^{2}\gamma \alpha + 7/17280bC^{3}\alpha
\\
&+ 1/1728bC\gamma C\alpha + 1/4320bC^{2}\gamma \alpha - 5/1728bC^{3} +
1/144b\gamma C\alpha + 1/360bC^{2} + 1/12b\gamma \alpha
\\*
&- 1/12bC - 1/2b\alpha + b
\\
b^{*} \mapsto s &:=1301/15552000C^{5}\gamma \alpha b^{*} + 2431/6220800
\gamma C^{4}\gamma \alpha b^{*} - 2507/3110400C^{4}\gamma \alpha b^{*}
- 53/518400\gamma C^{4}b^{*}
\\
&- 203/518400C^{4}\alpha b^{*} - 19/172800\gamma C^{3}\alpha b^{*} + 23/51840C^{3}
\gamma \alpha b^{*} + 199/259200C^{4} b^{*}
\\
&- 1/2160\gamma C^{3}\gamma \alpha b^{*} + 1/17280\gamma C^{3}b^{*} - 1/2880C^{2}
\gamma Cb^{*} + 1/1080C^{3}\alpha b^{*} + 1/216C^{3}b^{*}
\\
&+ 1/720C\gamma Cb^{*} - 1/720C^{2}\alpha b^{*} - 1/240\gamma ^{2}Cb^{*}
- 1/144C\gamma \alpha b^{*} + 1/24\gamma Cb^{*} - 1/12\gamma \alpha b^{*}
\\
&- 1/6Cb^{*} - 1/2\gamma b^{*} + b^{*}
\end{align*}
\begin{align*}
c \mapsto t &:= 51683/248832000cC^{6}\alpha - 162703/373248000cC^{5}
\gamma \alpha - 77/1036800c\gamma C^{4}\gamma - 1/77760cC^{3}\gamma C
\alpha
\\
&+ 2447/3110400cC^{4}\gamma \alpha + 37/124416cC^{5} + 17/1036800cC^{4}
\gamma + 203/259200cC^{4}\alpha
\\
&+ 23/86400c\gamma C^{3}\alpha - 73/259200cC^{2}\gamma C\alpha - 1/1350cC^{3}
\gamma \alpha - 463/518400cC^{4} + 7/8640c\gamma C^{3}
\\
&+ 7/17280cC^{2}\gamma C - 7/17280cC^{3}\alpha + 1/4320c(\gamma C)^{2}
- 1/2160c\gamma C^{2}\alpha - 1/2160cC\gamma C\alpha
\\
&+ 1/1080cC^{2}\gamma \alpha - 1/270cC^{3} + 1/720c\gamma C\gamma
\alpha + 1/720cC\gamma C + 1/240c\gamma C\gamma + 1/360c\gamma C
\alpha
\\
&+ 1/720cC\gamma \alpha + c
\\
c^{*} \mapsto u &:= 7853/18662400\gamma C^{5}\alpha c^{*} - 1157/3110400
\gamma C^{5}c^{*} + 5833/12441600C^{5}\gamma c^{*} - 59/12441600C^{5}
\alpha c^{*}
\\
&- 97/1036800\gamma C^{2}\gamma C\alpha c^{*} - 37/518400\gamma C^{4}c^{*}
- 17/16200C^{4} \alpha c^{*} - 49/259200\gamma C^{3}\alpha c^{*}
\\
&+ 7/34560C^{2}\gamma C\alpha c^{*} - 7/34560C^{3}\gamma \alpha c^{*} -
7/17280\gamma C^{2}\gamma \alpha c^{*} - 7/8640\gamma C^{3}c^{*} + 47/17280C^{3}
\gamma c^{*}
\\
&- 1/432C^{2}\gamma \alpha c^{*} + 1/432C^{3}c^{*} + 1/1440\gamma C
\gamma \alpha c^{*} + 1/1440\gamma C^{2}c^{*} + 1/480C^{2}\gamma c^{*}
- 1/240\gamma C\gamma c^{*}
\\
&- 1/360\gamma C\alpha c^{*} - 1/720C\gamma \alpha c^{*} - 1/24
\gamma ^{2}\alpha c^{*} - 1/24\gamma Cc^{*} - 1/24C\gamma c^{*} - 1/24C
\alpha c^{*} + 1/6\gamma ^{2}c^{*}
\\
&+ 1/2\gamma c^{*} + c^{*}
\end{align*}
\begin{align*}
d \mapsto v &:= -1/9953280b^{*}C^{5}\gamma \alpha b^{*} + 7/17280b^{*}C^{2}
\gamma Cb^{*} - 1/5760b^{*}C^{2}\gamma \alpha b^{*} + d
\\
d^{*} \mapsto w &:= 1/5760bC^{2}\gamma \alpha b^{*}d^{*} + 1/1728bC^{2}
\alpha b^{*}d^{*} + d^{*}
\\
e \mapsto x &:= 13/2073600ecC^{4}\alpha c^{*} + 13/345600ecC^{2}
\gamma C\alpha c^{*} + 13/172800ec(\gamma C)^{2}\alpha c^{*}
\\
&+ 19/172800ecC^{3}\gamma c^{*} + e
\\
e^{*} \mapsto y &:= -19/172800cC^{3}\gamma c^{*}e^{*} + 13/172800cC^{3}
\alpha c^{*}e^{*} + 1/6c\gamma c^{*}e^{*} + 1/2cc^{*}e^{*} + e^{*}
\\
f \mapsto z &:= 13/172800fecC\gamma C\alpha c^{*}e^{*} + f
\\
f^{*} \mapsto h &:= 13/345600ecC^{3}\alpha c^{*}e^{*}f^{*} - 13/172800ecC
\gamma C\alpha c^{*}e^{*}f^{*} + 1/2ee^{*}f^{*} + f^{*}
\end{align*}
One can check these denominators are indeed multiplies of only the primes
2, 3, and 5. This map descends to $\Lambda _{S'}(E_{8})$, since
$\varphi (r_{\text{mult}}) \in (r_{\text{add}})$, as shown in the appendix.
The descended map is an isomorphism as explained in the proof of Proposition~\ref{prop:_iso_Dn}.
\end{proof}

\begin{cor}
\label{cor5.8}
Let $R$ be a commutative ring. Then
$\Lambda _{R}(E_{8}) \cong \Pi _{R}(E_{8})$ if and only if
$2, 3, 5 \in R^{\times}$.
\end{cor}

\begin{proof}
If $2, 3, 5 \in R^{\times}$, then one can regard the coefficients of
$\varphi $ in Proposition~\ref{prop:_iso_E8} as living in $R$, giving the
isomorphisms. Corollary~\ref{cor:_no_iso_E67} establishes the converse.
\end{proof}

This completes the proof of the main theorem in the more general setting
of commutative rings.

\section{Applications}
\label{sec:_applications}

\subsection{Multiplicative Ginzburg DG algebras}
\label{ss:_dg_results}

The multiplicative Ginzburg dg-algebra is a differential graded algebra
whose zeroth homology is the multiplicative preprojective algebra. Etg\"{u}--Lekili
showed that this algebra naturally arises in 4-dimensional symplectic geometry;
its category of dg-modules is equivalent to the wrapped Fukaya category
of a plumbing of cotangent bundles of 2-spheres. Their equivalence factors
through the Chekanov-Eliashberg dg-algebra associated to the Legendrian
knot whose filling gives the generating (non-compact) Lagrangians. See
Section 2 of Etg\"{u}--Lekili \cite{EL17} for more details. In this section,
we apply our main theorem to the dg-setting.

We work over a field $k$, but all statements hold over a commutative ring
$R$ with the caveat that the condition $\text{char}(k)$ is good be replaced
with the condition all bad primes are invertible in $R$.

\begin{defn}
\label{defn6.1}
The \emph{Ginzburg dg-algebra} (with zero potential) for the quiver
$Q$, denoted $\mathcal{G}_{Q}^{\text{add}}$, is the free algebra
$k \overline{Q} *_{kQ_{0}} kQ_{0}[ s_{\text{add}}]$ with
$s_{\text{add}}$ a formal variable. The grading is given by
$|\overline{Q}|=0$ and $|s_{\text{add}}| = -1$. The differential $d$ is
a degree $1$ derivation defined on generators by
$d \mid _{\overline{Q}} \equiv 0$ and
$d(s_{\text{add}}) = r_{\text{add}}$. The
\emph{multiplicative Ginzburg dg-algebra} for $Q$, denoted
$\mathcal{G}_{Q}^{\text{mult}}$, is the localized free algebra
$k \overline{Q}[(1+a^{*}a)^{-1}]_{a \in Q_{1}} *_{kQ_{0}} kQ_{0}[ s_{
\text{mult}}]$ with grading $| \overline{Q} | =0$ and
$|s_{\text{mult}}| =-1$ and differential determined by
$d \mid _{\overline{Q}} \equiv 0$ and
$d(s_{\text{mult}}) = r_{\text{mult}}$.
\end{defn}

\begin{rem}
The multiplicative Ginzburg dg-algebra does \emph{not} agree with the derived multiplicative preprojective algebra of Etg\"{u}--Lekili \cite[page 779]{EL17}, but as explained in \cite[Remark 4.7]{KS20} these dg-algebras are quasi-isomorphic, which is sufficient for our results.
\end{rem}

\begin{rem}
\label{rem6.2}
Since $Q$ is star-shaped, $a^{*}a$ is nilpotent on homology in both the
Ginzburg and multiplicative Ginzburg dg-algebras, by Corollary~\ref{cor:_inverses}. Hence we can also invert $(1+a^{*}a)$ in
$\mathcal{G}_{Q}^{\text{add}}$ without changing its quasi-isomorphism class.
Together with the fact that
$r_{\text{mult}} = r_{\text{add}} + \text{(higher order terms)}$, one can
regard $\mathcal{G}_{Q}^{\text{mult}}$ as a filtered deformation of
$\mathcal{G}_{Q}^{\text{add}}$ using the descending path length filtration:
\begin{equation*}
\mathcal{F}_{\text{mult}}^{\bullet} := \quad \mathcal{F}_{\text{mult}}^{
\geq 0} :=\mathcal{G}_{Q}^{\text{mult}} \supset \cdots \supset
\mathcal{F}_{\text{mult}}^{\geq \ell} := (k \overline{Q})^{\geq \ell} *_{kQ_{0}}
kQ_{0}[ s_{\text{mult}}] \supset \cdots
\end{equation*}
where $k \overline{Q}^{\geq \ell}$ is the $k$-linear span of paths of length
at least $\ell $. Using the Rees construction, one can view this filtered
deformation as a formal deformation that lives in positive degree in
$\text{HH}^{2}(\mathcal{G}_{Q}^{\text{add}})$.
\end{rem}

\begin{thm}
\label{thm:_dg_main}
Let $Q$ be an ADE Dynkin quiver. Then
$\mathcal{G}_{Q}^{\text{mult}} \cong \mathcal{G}_{Q}^{\text{add}}$ are quasi-isomorphic
if and only if $\text{char}(k)$ is good for $Q$.
\end{thm}

In type D this result is \cite[Theorem 13]{EL17}. In type E the quasi-isomorphism
in good characteristic follows from the vanishing of positive degree classes
in $\text{HH}^{2}(\mathcal{G}_{Q}^{\text{add}})$ proven in
\cite[Section 5]{LU21}, which we explain also follows from our results
or \cite[Theorem 13.1.1]{Schedler16}. The obstruction in bad characteristic
is new in type E for $\text{char}(k) = 3, 5$.

\begin{proof}
Suppose there exists a quasi-isomorphism
$\mathcal{G}_{Q}^{\text{mult}} \cong \mathcal{G}_{Q}^{\text{add}}$. This
map identifies their zeroth homology
\begin{equation*}
\Lambda (Q) = H_{0}(\mathcal{G}_{Q}^{\text{mult}}) \cong H_{0}(
\mathcal{G}_{Q}^{\text{add}}) = \Pi (Q)
\end{equation*}
which by Corollary~\ref{cor:_no_iso_Dn} and Corollary~\ref{cor:_no_iso_E67} is only possible if $\text{char}(k)$ is good for
$Q$.

Conversely, the Ginzburg dg-algebra is 2-Calabi--Yau and concentrated in
non-positive degrees. Hence by Van den Bergh duality
\cite[Theorem 1]{VdB_duality},
\begin{equation*}
\text{HH}^{2}(\mathcal{G}_{Q}^{\text{add}} ) \cong \text{HH}_{0}(
\mathcal{G}_{Q}^{\text{add}} ) = \text{HH}_{0}( \Pi (Q))
\end{equation*}
which in good characteristic vanishes in positive degree either by
\cite[Theorem 13.1.1]{Schedler16} or indirectly by our isomorphisms
\begin{equation*}
\text{HH}_{0}( \Pi (Q)) \cong \text{HH}_{0}( \Lambda (Q)) \cong kQ_{0}.
\end{equation*}
We conclude that $\mathcal{G}_{Q}^{\text{mult}}$ is trivial as a
\emph{formal} deformation of $\mathcal{G}_{Q}^{\text{add}}$. Therefore, one
obtains a quasi-isomorphism \emph{after completing}
$\mathcal{G}_{Q}^{\text{mult}}$ and $\mathcal{G}_{Q}^{\text{add}}$ with respect
to the descending path-length filtrations
\begin{equation*}
\mathcal{F}_{\text{mult}}^{\geq \ell} := (k \overline{Q})^{\geq \ell} *_{kQ_{0}}
kQ_{0}[ s_{\text{mult}}] \quad \text{and} \quad \mathcal{F}_{\text{add}}^{
\geq \ell} := (k \overline{Q})^{\geq \ell} *_{kQ_{0}} kQ_{0}[ s_{
\text{add}}].
\end{equation*}
To prove the existence of a quasi-isomorphism
$\mathcal{G}_{Q}^{\text{mult}} \cong \mathcal{G}_{Q}^{\text{add}}$, Etg\"{u}--Lekili:
(I) truncate the completed quasi-isomorphism to an algebra isomorphism
of uncompleted algebras, (II) correct the truncation to a chain map, and
(III) use the strong convergence of the spectral sequences associated to
the filtrations \cite[Theorem 13]{EL17}. Note that
\cite[Lemma 12]{EL17} establishes the strong convergence of the spectral
sequence associated to $\mathcal{F}_{\text{mult}}^{\bullet}$. This lemma
is stated in type D, but their proof holds in type E essentially using
the fact that $\Lambda (Q)$ is finite-dimensional in this case.
\end{proof}

We want to emphasize that the key ideas are already in Etg\"{u}--Lekili.
But a central question, the triviality of certain deformation classes of
the Ginzburg dg-algebra, is detectable using the zeroth Hochschild homology
of the preprojective algebra and the multiplicative preprojective algebra.
Namely we have the following:

\begin{cor}
\label{cor6.4}
$\text{HH}_{0}(\Pi (Q)) = \text{HH}_{0}(\Lambda (Q))$ if and only if
$\mathcal{G}_{Q}^{\text{mult}}$ is a trivial deformation of
$\mathcal{G}_{Q}^{\text{add}}$ as a dg-algebra.
\end{cor}

We conclude this section by recording our zeroth Hochschild homology calculation
in the dg-setting.

\begin{prop}
\label{prop6.5}
For $Q$ ADE Dynkin, the space of formal deformations of
$\mathcal{G}_{Q}^{\text{mult}}$ as a dg-algebra is $kQ_{0}$.
\end{prop}

\begin{proof}
The dg-algebra $\mathcal{G}_{Q}^{\text{mult}}$ is 2-Calabi--Yau: Etg\"{u}--Lekili
realize its category of dg-modules as a wrapped Fukaya category of a Weinstein
4-manifold, known to be 2-Calabi--Yau by Ganatra
\cite[Theorem 1.3]{Ganatra12}. By (1) Van den Bergh duality, (2) the fact
that $\mathcal{G}_{Q}^{\text{mult}}$ is concentrated in non-positive degrees,
and (3) Propositions~\ref{prop:_HH_0_D_n} and \ref{prop:_HH_0_Em} we have
\begin{equation*}
\text{HH}^{2}(\mathcal{G}_{Q}^{\text{mult}})
\overset{\text{(1)}}{\cong} \text{HH}_{0}(\mathcal{G}_{Q}^{\text{mult}})
\overset{\text{(2)}}{=} \text{HH}_{0}( \Lambda _{k}(Q))
\overset{\text{(3)}}{=} kQ_{0}.
\end{equation*}
Since $\text{HH}^{3}(\mathcal{G}_{Q}^{\text{mult}}) =0$, these deformations
are unobstructed.
\end{proof}

\subsection{Symmetric Frobenius structures}
\label{ss:_Frobenius}

The goal of this section is to demonstrate that in cases where the additive and multiplicative preprojective algebras are non-isomorphic, they genuinely exhibit different behavior. For instance, for $R$ a commutative ring of $\text{char}(R) = 2$, $\Pi_{R}(E_8)$ is a symmetric Frobenius $R$-algebra by projecting to top degree while $\Lambda_{R}(E_8)$ is not symmetric Frobenius. 

\begin{defn}
Let $R$ be a commutative ring. Let $A$ be an $R$-algebra that is finitely-generated and projective as an $R$-module. A Frobenius structure on $A$ is an $R$-linear map $\lambda: A \rightarrow R$ that is non-degenerate in the sense that $\lambda(a \cdot -) : A \rightarrow R$ is non-zero for all non-zero $a \in A$. A Frobenius form is \emph{symmetric} if $\lambda(ab) = \lambda(ba)$ for all $a, b \in A$. We say $A$ is (symmetric) Frobenius if there exists a (symmetric) Frobenius structure on $A$. 
\end{defn}

If $\lambda$ is a Frobenius structure then there exists a unique automorphism $\eta: A \rightarrow A$ such that $\lambda(ab) = \lambda(b\eta(a))$ for all $a, b \in A$, called the \emph{Nakayama automorphism} with respect to $\lambda$. 

\begin{rem} (Change of Frobenius structure)
 If $A$ is Frobenius then the set of Frobenius structures on $A$ has a simply transitive action by the group $A^{\times}$ of invertible elements of $A$ given by $u \cdot \lambda =  \lambda(- \cdot u)$. If $\lambda$ has Nakayama automorphism $\eta$ then $u \cdot \lambda$ has Nakayama automorphism $u \eta u^{-1}$.
\end{rem}

\begin{prop} \label{prop: sym Frob mult}
$\Lambda_R(Q)$ is not symmetric Frobenius for any pair $(Q, R)$ with $Q \neq A_1$. 
\end{prop}

\begin{proof}
Let $A$ be any symmetric Frobenius algebra with form $\lambda: A \rightarrow R$. The symmetric condition says that $\lambda$ \emph{vanishes} on the space $[A, A]$. The non-degeneracy condition says that $\lambda$ does \emph{not} vanish on any non-trivial left ideal. Hence if $[A, A]$ contains a non-trivial left ideal then $A$ is not symmetric Frobenius. We will establish this for $A =\Lambda_R(Q)$ using the left ideal $\Lambda^{+}_R(Q)$ generated by the set of arrows, which is non-zero if $Q \neq A_1$.

By Corollaries \ref{cor:HH0_An_R}, \ref{cor:_HH0_Dn}, and \ref{cor:HH_0_Em_R},  $\Lambda_R(Q)_{\text{cyc}} = RQ_0$, which implies $[\Lambda_R(Q),  \Lambda_R(Q)]$ contains the ideal $\Lambda^{+}_R(Q)$, as desired.
\end{proof}

It is well-known that $\Pi_R(Q)$ is symmetric Frobenius if an only if $Q = A_1$ or $\text{char}(R) = 2$ and $Q = D_{2n}$, $E_7$, or $E_8$ for $n \geq 2$. For instance:
\begin{prop} \cite[Subsection 5.3]{Eu08}
Let $Q$ be ADE Dynkin. $\Pi_{\mathbb{Z}}(Q)$ is Frobenius over $\mathbb{Z}$ with Nakayama automorphism squaring to the identity. 
Moreover, if $Q = A_1$, $D_{2n}$, $E_7$, or $E_8$ for $n \geq 2$ then the Nakayama automorphism acts trivially on the vertices and sends each arrow $a$ to $\pm a$. Otherwise, the Nakayama automorphism acts non-trivially on vertices. 
\end{prop}

The Frobenius form projects to top degree and hence factors through $\Pi_{\mathbb{Z}}^{\text{top}}(Q)$ where $\text{top}$ is the maximal non-zero degree in $\Pi_{\mathbb{Z}}(Q)$. One can ask if this Frobenius form remains non-degenerate when viewed as a map $\lambda_{R} : \Pi_{R}(Q) \rightarrow R$ with $\text{char}(R) = 2$. The specific form of the Nakayama automorphism over $\mathbb{Z}$ for $Q = D_{2n}$, $E_7$, and $E_8$ implies that it is the identity over $R$. Hence $\lambda_R$ is a symmetric Frobenius form.  

Consequently, if $\text{char}(R) = 2$, and $Q = A_1$, $ D_{2n}$, $E_7$, or $E_8$, then the Nakayama automorphism is the identity map and $\Pi_R(Q)$ is symmetric Frobenius. We now demonstrate the converse, which is probably well-known.  

\begin{prop} 
Let $R$ be a commutative ring of $\text{char}(R) =2$, and let $Q$ be ADE Dynkin. Then $\Pi_{R}(Q)$ is symmetric Frobenius if and only if $Q = A_1$, $D_{2n}$, $E_7$, $E_8$ with $n \geq 2$.
\end{prop}

\begin{proof}
As in the first paragraph of the proof of Proposition \ref{prop: sym Frob mult}, we need to show that $[\Pi_R(Q), \Pi_R(Q)]$ contains a non-trivial left ideal. Any longest path $p \in \Pi_R(Q)$ gives a minimal left ideal $\langle p \rangle = R \cdot p$. So $\Pi_R(Q)$ is \emph{not} symmetric Frobenius if there exists a longest path $p \in [\Pi_R(Q), \Pi_R(Q)]$. Indeed, since the longest path $p$ has source $s(p) \in Q_0$ and target $t(p) = \eta(s(p)) \in Q_0$, we need $\eta \! \mid_{Q_0} = \text{id}_{Q_0}$ in order for $p$ to be a cycle and therefore not equal $[p, e_{t(p)}] \in [\Pi_R(Q), \Pi_R(Q)]$.
\end{proof}


\section{Appendix: Magma code for type E}
\label{sec:_appendix}

The following Magma code verifies that the given maps for type E descend to
the multiplicative preprojective algebra.

Note that we present $\Pi (E_{6})$ as a free algebra modulo the two-sided
ideal $I$. We added the additional generators alpha, beta, and gamma to
conform with the notation in the paper. Then we write the image of
$r_{\text{mult}}$ in $\Pi (E_{6})$ as rel\_mult and check that it lies in
the ideal $I$.
\begin{Verbatim}[fontsize = \small]
QQ := Rationals();
R_E6<a, b, c, d, e, da, db, dc, dd, de, beta, alpha, gamma>  := FreeAlgebra(QQ, 13);
Arrows := [ a, b, c, d, e]; DualArrows := [da, db, dc, dd, de];
Sources := [4, 2, 5, 1, 6]; Targets:= [3, 3, 3, 2, 5];
R_path := [Arrows[i]*Arrows[j]: i in [1..5], j in [1..5]|Targets[i] ne Sources[j]]
  cat [DualArrows[j]*DualArrows[i] : i in [1..5], j in [1..5]|Targets[i] ne Sources[j]]
  cat  [Arrows[i]*DualArrows[j] : i in [1..5], j in [1..5]|Targets[i] ne Targets[j]]
  cat [DualArrows[i]*Arrows[j] : i in [1..5], j in [1..5]|Sources[i] ne Sources[j]];
R_preproj := [d*dd, dd*d - b*db, gamma + beta + alpha, a*da, de*e -c*dc, e*de];
R_greek :=[ alpha-da*a, beta-db*b, gamma-dc*c];
I := ideal< R_E6 | R_path cat R_preproj cat R_greek>;
r := 1-1/2*alpha -1/12*alpha*gamma -1/24*alpha*gamma^2;
s := 1-1/2*gamma -1/6*alpha*gamma;
t := 1- 1/12*gamma*alpha;
u := 1+1/2*gamma + 1/12*gamma*alpha - 1/24*alpha*gamma*alpha;
rel_mult := (b*r*s*db - dd*d) + (c*t*u*dc - de*e)
    + (alpha+s*beta*r+u*gamma*t +u*gamma*t*s*beta*r);
rel_mult in I;
\end{Verbatim}

Let
$\varphi :\mathbb{Z}[1/2, 1/3]\overline{E_{6}} \rightarrow \Pi _{
\mathbb{Z}[1/2, 1/3]}(E_{6})$ be the map sending $b \mapsto br$,
$b^{*} \mapsto sb^{*}$, $c \mapsto ct$, $c^{*} \mapsto uc^{*}$, that is
the identity on all vertices and other arrows. Since
$\varphi (r_{\text{mult}})$ is in the ideal generated by
$r_{\text{add}}$, it follows that this map descends to a map
$\Lambda _{\mathbb{Z}[1/2, 1/3]}(E_{6}) \rightarrow \Pi _{\mathbb{Z}[1/2,
1/3]}(E_{6})$. Since $r, s, t$ and $u$ are invertible, this map is an isomorphism.
The code and argument for $E_{7}$ is similar:
\begin{Verbatim}[fontsize = \small]
QQ := Rationals();
R_E7<a, b, c, d, e, f, da, db, dc, dd, de, df, beta, gamma, alpha> := FreeAlgebra(QQ, 15);
Arrows := [ a, b, c, d, e, f]; DualArrows := [da, db, dc, dd, de, df];
Sources := [4, 2, 5, 1, 6, 7]; Targets:= [3, 3, 3, 2, 5, 6];
R_path := [ Arrows[i]*Arrows[j]: i in [1..6], j in [1..6] | Targets[i] ne Sources[j]]
  cat [DualArrows[j]*DualArrows[i] : i in [1..6], j in [1..6] | Targets[i] ne Sources[j]]
  cat  [Arrows[i]*DualArrows[j] : i in [1..6], j in [1..6] | Targets[i] ne Targets[j]]
  cat [DualArrows[i]*Arrows[j] : i in [1..6], j in [1..6] | Sources[i] ne Sources[j]];
R_preproj := [d*dd, dd*d - b*db, dc*c+db*b+da*a, a*da, de*e -c*dc, df*f -e*de, f*df];
R_greek :=[ alpha-da*a, beta-db*b, gamma-dc*c];
I := ideal< R_E7 | R_path cat R_preproj cat R_greek>;
x := 1/4*e*c*dc + e; y := 1/4*c*dc*de + de;
r :=  -1/864*b*(alpha*gamma)^3*alpha + 1/432*b*(alpha*gamma)^3
    + 1/288*b*(alpha*gamma)^2*alpha - 1/144*b*(gamma*alpha)^2
    - 1/144*b*alpha*gamma^2*alpha - 1/144*b*(alpha*gamma)^2
    + 1/24*b*alpha*gamma*alpha - 1/12*b*alpha*gamma
    - 1/2*b*alpha + b;
s :=  1/144*(alpha*gamma)^2*db - 1/6*alpha*gamma*db - 1/2*gamma*db + db;
t := -1/6912*c*(alpha*gamma)^3*alpha + 1/432*c*(alpha*gamma)^2*gamma*alpha
    + 7/3456*c*(alpha*gamma)^3 - 1/288*c*gamma*alpha*gamma^2*alpha
    - 1/288*c*(alpha*gamma)^2*gamma + 1/144*c*gamma*alpha*gamma^2
    - 1/72*c*alpha*gamma^2*alpha - 1/144*c*(alpha*gamma)^2
    + 1/24*(c*gamma*alpha*gamma + c*alpha*gamma^2 + c*alpha*gamma*alpha)
    + 1/12*c*gamma^2 - 1/12*c*gamma*alpha + 1/2*c*gamma + c;
u := -1/3456*gamma*alpha*gamma*alpha*gamma*alpha*dc + 1/12*gamma^2*dc
    + 1/12*gamma*alpha*dc + dc;
rel_mult := (r*s - dd*d) + (t*u -y*x) + (x*y - df*f)
    +(alpha+s*r+u*t +u*t*s*r);
rel_mult in I;
\end{Verbatim}

The calculation for $E_{8}$ is similar, but in order to produce code that
runs in under a minute we define an additional truncation function to remove
terms of length greater than 28, which are necessarily zero in
$\Pi (E_{8})$.
\begin{Verbatim}[fontsize = \small]
Truncate := function(p, d);
L := Terms(p);  i:=1;
while i le #L do
    if Degree(L[i]) gt d then L := Remove(L, i); i := i-1; end if;
    i := i+1;
end while;
return &+([0] cat L); end function;

QQ := Rationals();
R_E8<a, b, c, d, e, f, g, da, db, dc, dd, de, df, dg, beta, gamma, alpha> := FreeAlgebra(QQ, 17);
Arrows := [ a, b, c, d, e, f, g]; DualArrows := [da, db, dc, dd, de, df, dg];
Sources := [4, 2, 5, 1, 6, 7, 8]; Targets:= [3, 3, 3, 2, 5, 6, 7];
R_path := [ Arrows[i]*Arrows[j]: i in [1..7], j in [1..7] | Targets[i] ne Sources[j]]
    cat [ DualArrows[j]* DualArrows[i] : i in [1..7], j in [1..7] | Targets[i] ne Sources[j]]
    cat  [ Arrows[i]*DualArrows[j] : i in [1..7], j in [1..7] | Targets[i] ne Targets[j] ]
    cat [DualArrows[i]*Arrows[j] : i in [1..7], j in [1..7] | Sources[i] ne Sources[j] ];
R_preproj := [d*dd, dd*d - b*db, dc*c+db*b+da*a, a*da, de*e -c*dc, df*f -e*de, f*df-dg*g, g*dg];
R_greek :=[ alpha-da*a, beta-db*b, gamma-dc*c];
I := ideal< R_E8 | R_path cat R_preproj cat R_greek>;

C := alpha*gamma; //To visually condense the polynomials
p := 53383/93312000*a*gamma*C^6 + 461/1036800*a*gamma*C^2*gamma*C*alpha
    + 199/129600*a*gamma*C^3*alpha + 1/4320*a*gamma*C^2*gamma*alpha + a;
q := -53383/93312000*gamma*C^6*da - 26639/24883200*C^6*da + 259/691200*C^5*gamma*da
    + 1/2073600*C^5*da + 91/518400*C^4*gamma*da -1/3456*C^2*gamma*C*da + 7/17280*C^3*gamma*da
    + 11/2880*C^3*da - 1/720*C*gamma*C*da + 1/720*C^2*gamma*da +1/144*C^2*da + 1/12*C*da + da;
r := 773/9331200*b*C^6*alpha - 307/2488320*b*C^5*alpha - 17/172800*b*C^3*gamma*C*alpha
    + 2261/3110400*b*C^4*gamma*alpha + 749/3110400*b*C^5 + 301/1036800*b*gamma*C^2*gamma*C*alpha
    + 371/518400*b*C^4*alpha - 1/17280*b*gamma*C^2*gamma*C + 17/103680*b*C^2*gamma*C*alpha
    - 37/57600*b*C^3*gamma*alpha - 13/28800*b*C^4 + 11/17280*b*gamma*C^2*gamma*alpha
    + 7/17280*b*C^3*alpha + 1/1728*b*C*gamma*C*alpha + 1/4320*b*C^2*gamma*alpha - 5/1728*b*C^3
    + 1/144*b*gamma*C*alpha + 1/360*b*C^2 + 1/12*b*gamma*alpha - 1/12*b*C - 1/2*b*alpha + b;
s :=1301/15552000*C^5*gamma*alpha*db + 2431/6220800*gamma*C^4*gamma*alpha*db
    - 2507/3110400*C^4*gamma*alpha*db - 53/518400*gamma*C^4*db - 203/518400*C^4*alpha*db
    - 19/172800*gamma*C^3*alpha*db + 23/51840*C^3*gamma*alpha*db + 199/259200*C^4 *db
    - 1/2160*gamma*C^2*gamma*alpha*db + 1/17280*gamma*C^3*db - 1/2880*C^2*gamma*C*db
    + 1/1080*C^3*alpha*db + 1/216*C^3*db + 1/720*C*gamma*C*db - 1/720*C^2*alpha*db
    - 1/240*gamma^2*C*db - 1/144*C*gamma*alpha*db + 1/24*gamma*C*db - 1/12*gamma*alpha*db
    - 1/6*C*db - 1/2*gamma*db + db;
t := 51683/248832000*c*C^6*alpha - 162703/373248000*c*C^5*gamma*alpha
    - 77/1036800*c*gamma*C^4*gamma - 1/77760*c*C^3*gamma*C*alpha
    + 2447/3110400*c*C^4*gamma*alpha + 37/124416*c*C^5 + 17/1036800*c*C^4*gamma
    + 203/259200*c*C^4*alpha + 23/86400*c*gamma*C^3*alpha - 73/259200*c*C^2*gamma*C*alpha
    - 1/1350*c*C^3*gamma*alpha - 463/518400*c*C^4 + 7/8640*c*gamma*C^3 + 7/17280*c*C^2*gamma*C
    - 7/17280*c*C^3*alpha + 1/4320*c*(gamma*C)^2 - 1/2160*c*gamma*C^2*alpha
    - 1/2160*c*C*gamma*C*alpha + 1/1080*c*C^2*gamma*alpha - 1/270*c*C^3
    + 1/720*c*gamma*C*gamma*alpha + 1/720*c*C*gamma*C + 1/240*c*gamma*C*gamma
    + 1/360*c*gamma*C*alpha + 1/720*c*C*gamma*alpha + c;
u := 7853/18662400*gamma*C^5*alpha*dc - 1157/3110400*gamma*C^5*dc + 5833/12441600*C^5*gamma*dc
    - 59/12441600*C^5*alpha*dc - 97/1036800*gamma*C^2*gamma*C*alpha*dc - 37/518400*gamma*C^4*dc
    - 17/16200*C^4 *alpha*dc - 49/259200*gamma*C^3*alpha*dc + 7/34560*C^2*gamma*C*alpha*dc
    - 7/34560*C^3*gamma*alpha*dc - 7/17280*gamma*C^2*gamma*alpha*dc - 7/8640*gamma*C^3*dc
    + 47/17280*C^3*gamma*dc - 1/432*C^2*gamma*alpha*dc + 1/432*C^3*dc
    + 1/1440*gamma*C*gamma*alpha*dc + 1/1440*gamma*C^2*dc + 1/480*C^2*gamma*dc
    - 1/240*gamma*C*gamma*dc - 1/360*gamma*C*alpha*dc - 1/720*C*gamma*alpha*dc
    - 1/24*gamma^2*alpha*dc - 1/24*gamma*C*dc - 1/24*C*gamma*dc - 1/24*C*alpha*dc
    + 1/6*gamma^2*dc + 1/2*gamma*dc + dc;
v := -1/9953280*d*b*C^5*gamma*alpha*db + 7/17280*d*b*C^2*gamma*C*db
    - 1/5760*d*b*C^2*gamma*alpha*db + d;
w := 1/5760*b*C^2*gamma*alpha*db*dd + 1/1728*b*C^2*alpha*db*dd + dd;
x := 13/2073600*e*c*C^4*alpha*dc + 13/345600*e*c*C^2*gamma*C*alpha*dc
    + 13/172800*e*c*(gamma*C)^2*alpha*dc + 19/172800*e*c*C^3*gamma*dc + e;
y :=  -19/172800*c*C^3*gamma*dc*de + 13/172800*c*C^3*alpha*dc*de
    + 1/6*c*gamma*dc*de + 1/2*c*dc*de + de;
z := 13/172800*f*e*c*C*gamma*C*alpha*dc*de + f;
h := 13/345600*e*c*C^3*alpha*dc*de*df - 13/172800*e*c*C*gamma*C*alpha*dc*de*df + 1/2*e*de*df + df;

rel_mult := (v*w) + (r*s - w*v) + (p*q) + (t*u -y*x) + (x*y - h*z) + (z*h - dg*g) +(g*dg)
    + Truncate(q*p+s*r+u*t +Truncate(u*t, 14)*Truncate(s*r, 14), 16);
rel_mult in I;
\end{Verbatim}


\end{document}